\documentclass[11pt,reqno]{amsart}
\usepackage{xypic,graphics,hyperref}

\DeclareMathOperator{\bond}{\mathsf{bo}}
\DeclareMathOperator{\circuit}{\mathsf{ci}}
\DeclareMathOperator{\colspace}{colspace}
\DeclareMathOperator{\coker}{coker}

\DeclareMathOperator{\Cut}{Cut}

\DeclareMathOperator{\Flow}{Flow}
\DeclareMathOperator{\Hom}{Hom}
\DeclareMathOperator{\im}{im}

\DeclareMathOperator{\rank}{rank}
\DeclareMathOperator{\Span}{span}
\DeclareMathOperator{\Supp}{Supp}
\DeclareMathOperator{\supp}{supp}

\newcommand{\includefigure}[3]{{
  \begin{center}
  \resizebox{#1}{#2}{\includegraphics{{#3}}}
  \end{center}}}

\newtheorem{theorem}{Theorem}
\numberwithin{theorem}{section}
\newtheorem{proposition}[theorem]{Proposition}

\newtheorem{corollary}[theorem]{Corollary}
\newtheorem{lemma}[theorem]{Lemma}
\theoremstyle{definition}
\newtheorem{definition}[theorem]{Definition}
\newtheorem{remark}[theorem]{Remark}
\newtheorem{example}[theorem]{Example}

\newcommand{\excise}[1]{}
\newcommand{\Ldu}{L^{\mathrm{du}}}
\newcommand{\Lud}{L^{\mathrm{ud}}}
\newcommand{\CC}{\mathcal{C}}
\newcommand{\FF}{\mathcal{F}}
\newcommand{\LL}{\mathcal{L}}
\newcommand{\MM}{\mathcal{M}}
\newcommand{\RR}{\mathcal{R}}

\newcommand{\HH}{\tilde{H}}
\newcommand{\Qq}{\mathbb{Q}}
\newcommand{\Rr}{\mathbb{R}}
\newcommand{\Zz}{\mathbb{Z}}

\newcommand{\betti}{\tilde\beta}
\newcommand{\bd}{\partial}
\newcommand{\cbd}{\partial^*}

\newcommand{\isom}{\cong}
\newcommand{\x}{\times}
\newcommand{\scp}[1]{\left\langle#1\right\rangle} 
\newcommand{\st}{\colon}
\newcommand{\sm}{\setminus}
\newcommand{\0}{\emptyset}
\newcommand{\tor}{{\bf T}} 
\newcommand{\torh}{{\bf t}}

\newcommand{\uncalcut}{\bar\chi}  
\newcommand{\calcut}[2]{\chi_{#1}(#2)} 
\newcommand{\calcutta}[2]{\chi(#1,#2)} 
\newcommand{\uncalflow}{\varphi} 
\newcommand{\calflow}{\hat\varphi} 

\begin{document}
\title{Cuts and Flows of Cell Complexes}

\author{Art M.\ Duval}
\address{Department of Mathematical Sciences\\ University of Texas at El Paso}
\email{artduval@math.utep.edu}
\author{Caroline J.\ Klivans}
\address{Departments of Applied Mathematics, Computer Science, and Mathematics \\ Brown University}
\email{klivans@brown.edu}
\author{Jeremy L.\ Martin}
\address{Department of Mathematics \\ University of Kansas}
\email{jlmartin@ku.edu}
\date{\today}
\thanks{Corresponding author: Art Duval, email artduval@math.utep.edu.\\
Third author supported in part by a Simons Foundation Collaboration Grant and by National Security Agency grant no.~H98230-12-1-0274. It is our pleasure to thank Andrew Berget, John Klein, Russell Lyons, Ezra Miller, Igor Pak, Dave Perkinson, and Victor Reiner for valuable discussions, some of which took place at the 23rd International Conference on Formal Power Series and Algebraic Combinatorics (Reykjavik, 2011).  We are also grateful for the suggestions from an anonymous referee.}


\begin{abstract}
We study the vector spaces and integer lattices of cuts and flows
associated with an arbitrary finite CW complex, and their
relationships to group invariants including the
critical group of a complex.  Our results extend to higher dimension the
theory of cuts and flows in graphs, most notably the work of Bacher,
de~la~Harpe and Nagnibeda. We construct explicit bases for the cut and flow spaces,
interpret their coefficients topologically, and give sufficient
conditions for them to be integral bases of the cut and flow lattices. 
Second, we determine the precise relationships between the discriminant groups
of the cut and flow lattices and the higher critical and cocritical groups with error terms
corresponding to torsion (co)homology.  As an application, we generalize a result of Kotani and
Sunada to give bounds for the complexity, girth, and connectivity of a complex in
terms of Hermite's constant.
\end{abstract}
\maketitle

\section{Introduction}
This paper is about vector spaces and integer lattices of cuts and
flows associated with a finite cell complex.
Our primary motivation is the study of critical groups of cell
complexes and related group invariants.  The critical group of a graph
 is a finite abelian group whose order
is the number of spanning forests.  The definition was introduced
independently in several different settings, including arithmetic
geometry~\cite{Lor}, physics \cite{Dhar}, and algebraic
geometry~\cite{BHN} (where it is also known as the Picard group or
Jacobian group).  It has received considerable recent attention for
its connections to discrete dynamical systems, tropical geometry, and
linear systems of curves; see, e.g.,~\cite{BN,Biggs-Chip,BL,HMY}.

In previous work~\cite{Critical}, the authors extended the definition
of the critical group to a cell complex~$\Sigma$ of arbitrary
dimension.  To summarize, the critical group $K(\Sigma)$ can be
calculated using a reduced combinatorial Laplacian, and its order is a
weighted enumeration of the cellular spanning trees of~$\Sigma$. 
Moreover, the action of the critical group on cellular $(d-1)$-cochains
gives a model of discrete flow on~$\Sigma$, generalizing the
chip-firing and sandpile models; see, e.g., \cite{Biggs-Chip,Dhar}.

Bacher, de la Harpe, and Nagnibeda first defined the lattices $\CC$
and $\FF$ of integral cuts and flows for a graph~\cite{BHN}.  By
regarding a graph as an analogue of a Riemann surface, they
interpreted the discriminant groups $\CC^\sharp/\CC$ and
$\FF^\sharp/\FF$ respectively as the Picard group of divisors and as
the Jacobian group of holomorphic forms.  In particular, they showed
that the critical group $K(G)$ is isomorphic to both $\CC^\sharp/\CC$
and $\FF^\sharp/\FF$.  Similar definitions and results appear in the work of
Biggs~\cite{Biggs-Chip}.

In the present paper, we define the \emph{cut and flow spaces} and
\emph{cut and flow lattices} of a cell complex $\Sigma$ by
\begin{align*}
\Cut(\Sigma) &= \im_\Rr\cbd, & \Flow(\Sigma) &= \ker_\Rr\bd,\\
\CC(\Sigma) &= \im_\Zz\cbd, & \FF(\Sigma) &= \ker_\Zz\bd,
\end{align*}
where $\bd$ and $\cbd$ are the top cellular boundary and coboundary
maps of $\Sigma$.  In topological terms, cut- and flow-vectors are
cellular coboundaries and cycles, respectively.  Equivalently, the
vectors in $\Cut(\Sigma)$ support sets of facets whose deletion
increases the codimension-1 Betti number, and the vectors in
$\Flow(\Sigma)$ support nontrivial rational homology classes.

In the higher-dimensional setting, the groups $\CC^\sharp/\CC$ and $\FF^\sharp/\FF$
are not necessarily isomorphic to each other.  Their precise relationship involves
several other groups: the critical group $K(\Sigma)$, a dually defined \emph{cocritical
  group} $K^*(\Sigma)$, and the \emph{cutflow group}
$\Zz^n/(\CC\oplus\FF)$.  We show that
the critical and cocritical groups are respectively isomorphic to the discriminant groups of the
cut lattice and flow lattice, and that the cutflow group mediates between them
with an ``error term'' given by homology.
Specifically, if $\dim \Sigma = d$, then we have the short exact sequences
\begin{align*}
&0 \to \Zz^n / (\CC\oplus\FF) \to \CC^\sharp/\CC\isom K(\Sigma) \to \tor(\HH_{d-1}(\Sigma;\Zz))\to 0,\\
&0 \to \tor(\HH_{d-1}(\Sigma;\Zz)) \to \Zz^n/(\CC\oplus\FF)\to \FF^\sharp/\FF \isom K^*(\Sigma) \to 0
\end{align*}
(Theorems~\ref{new-cut-theorem} and~\ref{new-flow-theorem})
where $\tor$ denotes the torsion summand.  The sizes of these groups are then given by
\begin{alignat*}{2} 
|\CC^\sharp/\CC| = |K(\Sigma)| &=  \tau(\Sigma)  & &= \tau^*(\Sigma)\cdot\torh^2,\\
|\FF^\sharp/\FF| = |K^*(\Sigma)| &=  \tau^*(\Sigma) & &= \tau(\Sigma)/\torh^2,\\
|\Zz^n/(\CC\oplus\FF)| &=  \tau(\Sigma)/\torh  & &= \tau^*(\Sigma)\cdot\torh,
\end{alignat*}
(Theorems~\ref{all-counts} and~\ref{cocritical-count}),
where $\torh=|\tor(\HH_{d-1}(\Sigma;\Zz))|$ and $\tau(\Sigma)$ and $\tau^*(\Sigma)$
are the weighted enumerators
\begin{align*}
\tau(\Sigma)&=\sum_\Upsilon|\tor(\HH_{d-1}(\Upsilon;\Zz))|^2,&
\tau^*(\Sigma)&=\sum_\Upsilon|\tor(\HH_d(\Omega,\Upsilon;\Zz))|^2,
\end{align*}
where both sums run over all cellular spanning forests $\Upsilon\subseteq\Sigma$
(see equation~\eqref{complexity})
and $\Omega$ is an acyclization of $\Upsilon$
(see Definition~\ref{acyclization}).

Before proving these results, we study the cut space (Section~\ref{cutspace-section}),
the flow space (Section~\ref{flowspace-section}), and the cut and flow lattices
(Section~\ref{integral-bases-section}) in some detail.  In order to do this, we
begin in Section~\ref{cellular-forest} by describing and enumerating cellular spanning forests of an arbitrary cell complex, generalizing our earlier work \cite{Simplicial,Cellular}.  
Similar results were independently achieved, using different techniques, by Catanzaro, Chernyak and Klein~\cite{Catanzaro}.
Our methods and results are very close to those of Lyons~\cite{Lyons}, but our technical emphasis is slightly different.

Every cellular spanning forest $\Upsilon$ naturally gives rise to
bases of the cut space (Theorem~\ref{cut-basis}) and the flow space
(Theorem~\ref{flow-basis}).  In the graphic case, these basis vectors
are simply signed characteristic vectors of fundamental cocircuits and
circuits in the graphic matroid, and they always form integral bases
for the cut and flow lattices.  For a general cellular complex, the
\emph{supports} of basis vectors are given by cocircuits and circuits in the cellular
matroid of $\Sigma$ (i.e., the matroid represented by the columns
of~$\bd$), but their \emph{entries} are not determined by the matroid.  We prove that the basis
vectors can be scaled so that their entries are torsion coefficients
of homology groups of certain subcomplexes (Theorems~\ref{calibration}
and~\ref{charflow}).  Under certain conditions on $\Upsilon$, these
bases are in fact integral bases for the cut and flow lattices
(Theorems~\ref{integral-basis-cut} and~\ref{integral-basis-flow}).
Although the matroid data alone is not enough to extend the theory of \cite{BHN}
to arbitrary cell complexes, the perspective of matroid theory will frequently be useful.

The idea of studying cuts and flows of matroids goes back to
Tutte~\cite{Tutte}.  More recently, Su and Wagner~\cite{SW} define
cuts and flows of a regular matroid (i.e., one represented by a
totally unimodular matrix~$M$); when $M$ is the boundary matrix of a
cell complex, this is the case where
the torsion coefficients are all trivial.
Su and Wagner's definitions
coincide with ours; their focus, however, is on recovering the
structure of a matroid from the metric data of its flow lattice.

In the final section of the paper, we generalize a theorem of Kotani
and Sunada~\cite{KS}, who observed that a classical inequality for
integer lattices, involving Hermite's constant (see, e.g.,
\cite{Lag}), could be applied to the flow lattice of a graph to give a
bound for girth and complexity.  We prove the corresponding result for cell
complexes (Theorem~\ref{Hermite-relations}), where ``girth'' means the
size of a smallest circuit in the cellular matroid (or, topologically,
the minimum number of facets supporting a nonzero homology class) and
``complexity'' is the torsion-weighted count of cellular spanning
trees.

\section{Preliminaries} \label{prelims}
In this section we review the tools needed throughout the
paper: cell complexes, cellular spanning trees and forests, integer
lattices, and matroids.

\subsection{Cell complexes} \label{cell}

 Our work is motivated by algebraic graph
theory, including critical groups, cut and flow spaces and lattices,
and the chip-firing game. Our central goal is to extend the theory
from graphs to higher-dimensional spaces.  Thus we work in the setting
of a finite CW complex, regarded as the higher-dimensional analogue of
a graph.  Accordingly, we begin by reviewing some of the topology of
cell complexes; for a general reference, see~\cite[p.~5]{Hatcher}.  The
reader more familiar with simplicial complexes may safely consider
that special case throughout.

Throughout the paper, $\Sigma$ will denote a finite CW complex
(which we refer to simply as a cell complex) of
dimension~$d$.  We adopt the convention that~$\Sigma$
has a unique cell of dimension~$-1$ (as though it were an abstract
simplicial complex); this will allow our results to
specialize correctly to the case $d=1$ (i.e., that~$\Sigma$ is a
graph).  We write~$\Sigma_i$ for the set of~$i$-dimensional cells
in~$\Sigma$, and $\Sigma_{(i)}$ for the $i$-dimensional skeleton of
$\Sigma$, i.e., $\Sigma_{(i)}=\Sigma_i\cup\Sigma_{i-1}\cup\cdots\cup\Sigma_0$.
Again, in keeping with simplicial-complex terminology,
a cell of dimension $d$ is called a \emph{facet}.

Unless otherwise stated, every $d$-dimensional subcomplex
$\Gamma\subseteq\Sigma$ will be assumed to have a full codimension-1
skeleton, i.e., $\Gamma_{(d-1)}=\Sigma_{(d-1)}$.  Accordingly, for
simplicity of notation, we will often make no distinction between the
subcomplex $\Gamma$ itself and its set $\Gamma_d$ of facets.

The symbol $C_i(\Sigma)=C_i(\Sigma;R)$ denotes the group of~$i$-dimensional
cellular chains with coefficients in a ring~$R$.  The
$i$-dimensional cellular boundary and coboundary maps are
respectively
$\bd_i(\Sigma;R)\colon C_i(\Sigma;R) \to C_{i-1}(\Sigma;R)$
and $\cbd_i(\Sigma;R)\colon C_{i-1}(\Sigma;R) \to C_i(\Sigma;R)$;
we will write simply $\bd_i$ and $\cbd_i$ whenever possible.

When $\Sigma$ is a graph (i.e., a cell complex of dimension~1), its top boundary map is a familiar object, namely its signed vertex-edge incidence matrix (with respect to some edge orientation).  
In this article, our goal will be to extract combinatorial information about an \emph{arbitrary} cell complex from its top-dimensional boundary map (which can be any integer matrix).

The $i^{th}$ reduced cellular homology and cohomology groups of $\Sigma$ are
respectively $\HH_i(\Sigma;R)=\ker\bd_i/\im\bd_{i+1}$ and
$\HH^i(\Sigma;R)=\ker\cbd_{i+1}/\im\cbd_i$.  We say that $\Sigma$ is
\emph{$R$-acyclic in codimension~one} if $\HH_{d-1}(\Sigma;R)=0$.  For a
graph ($d=1$), both $\Qq$- and $\Zz$-acyclicity in codimension~one are
equivalent to connectedness.  The $i^{th}$ reduced Betti number is
$\betti_i(\Sigma)=\dim\HH_i(\Sigma;\Qq)$, and the $i^{th}$ torsion
coefficient $\torh_i(\Sigma)$ is the cardinality of the torsion subgroup $\tor(\HH_i(\Sigma;\Zz))$.
We will frequently use the fact that
\begin{equation} \label{UCTC}
\tor(\HH_{d-1}(\Sigma;\Zz)) \isom \tor(\HH^d(\Sigma;\Zz))
\end{equation}
which is a special case of the universal coefficient theorem for cohomology \cite[p.~205, Corollary~3.3]{Hatcher}.
A pair of complexes $\Gamma\subseteq\Sigma$
induces a relative complex $(\Sigma,\Gamma)$, with relative homology and cohomology
$\HH_i(\Sigma,\Gamma;R)$ and $\HH^i(\Sigma,\Gamma;R)$ and torsion coefficients
$\torh_i(\Sigma,\Gamma)=|\tor(\HH_i(\Sigma,\Gamma;\Zz))|$.

While many definitions and results can be stated purely algebraically (e.g., in terms of chain complexes over $\Zz$), we regard the underlying object of interest as the cell complex (see Remark~\ref{puncturing}).

\subsection{Spanning Forests and Laplacians} \label{tree-Lap-subsection}
Our work on cuts and flows will use the theory of spanning forests
in arbitrary dimensions.  
Define a \emph{cellular spanning forest} (CSF) of $\Sigma$ to be a subcomplex
$\Upsilon\subseteq \Sigma$ such that $\Upsilon_{(d-1)}=\Sigma_{(d-1)}$ and
  \begin{subequations}
  \begin{align}
  \label{acyc-condn}  & \HH_d(\Upsilon;\Zz) = 0,\\
  \label{conn-condn}  & \rank\HH_{d-1}(\Upsilon;\Zz)=\rank\HH_{d-1}(\Sigma;\Zz), \quad\text{and}\\
  \label{count-condn} & |\Upsilon_d| = |\Sigma_d| - \betti_d(\Sigma) 
  \end{align}
  \end{subequations}
These conditions generalize the definition
of a spanning forest\footnote{That is, a maximal acyclic subgraph of $G$, not merely an acyclic subgraph containing all vertices.}
 of a graph~$G$: respectively, it is acyclic, has $c$ components, and has
$n-c$ edges, where $n$ and $c$ are the numbers of vertices and components of~$G$.
Just as in the graphic case, any two of the
conditions (\ref{acyc-condn}), (\ref{conn-condn}), (\ref{count-condn})
together imply the third; 
the proof is just a slight modification of the proof of~\cite[Proposition 3.5]{Simplicial}.
  An equivalent and perhaps simpler definition is that
a subcomplex $\Upsilon\subseteq\Sigma$ is a cellular spanning forest if and only if its $d$-cells correspond to a column basis for the cellular boundary matrix $\bd=\bd_d(\Sigma)$; however, the definition focusing on integral homology is frequently the most useful (see, e.g., Remark~\ref{torsion-in-simplicial-case}).

In the case that $\Sigma$ is $\Qq$-acyclic in codimension~one, this
definition specializes to our earlier definition of a cellular
spanning tree~\cite[Definition~2.2]{Cellular}.

There are two main reasons that enumeration of spanning forests of
cell complexes is more complicated than for graphs.  First, many
properties of graphs can be studied component by component, so that
one can usually make the simplifying assumption of connectedness; on
the other hand, a higher-dimensional cell complex cannot in general be
decomposed into disjoint pieces that are all acyclic in codimension
one.  Second, for complexes of dimension greater than or equal to two,
the possibility of torsion homology affects enumeration.

Define the $i^{th}$ \textit{up-down},
\textit{down-up} and \textit{total} Laplacian operators\footnote{These
    are discrete versions of the Laplacian operators on
    differential forms of a Riemannian manifold.  The interested
    reader is referred to \cite{Eckmann} and \cite{Dodziuk} for their
    origins in differential geometry and, e.g.,
    \cite{Denham,Friedman,Merris} for more recent appearances in
    combinatorics.}
on~$\Sigma$ by
\begin{align*}
\Lud_i&=\bd_{i+1}\cbd_{i+1}\colon C_i(\Sigma;R) \to C_i(\Sigma;R),\\
\Ldu_i&=\cbd_i\bd_i\colon C_i(\Sigma;R) \to C_i(\Sigma;R),\\
L^{\textrm{tot}}_i&= \Lud_i + \Ldu_i.
\end{align*}
Moreover, define the \emph{complexity} of~$\Sigma$ as
\begin{equation} \label{complexity}
\tau(\Sigma)=\tau_d(\Sigma)=\sum_{\text{CSFs }\Upsilon\subseteq \Sigma}|\tor(\HH_{d-1}(\Upsilon;\Zz))|^2.
\end{equation}
The \emph{cellular matrix-tree theorem}~\cite[Theorem~2.8]{Cellular}
states that if $\Sigma$ is $\Qq$-acyclic in codimension~one
and
$L_{\bar{\Upsilon}}$ is the submatrix of $\Lud_{d-1}(\Sigma)$ obtained by deleting the
rows and columns corresponding to the facets of a $(d-1)$-spanning tree
$\Upsilon$, then
\begin{equation*} \label{CMTT}
\tau(\Sigma)=\frac{|\tor(\HH_{d-2}(\Sigma;\Zz))|^2}{|\tor(\HH_{d-2}(\Upsilon;\Zz))|^2}\det L_{\bar{\Upsilon}}.
\end{equation*}
In Section~\ref{cellular-forest}, we will generalize this formula to
arbitrary cell complexes (i.e., not requiring that~$\Sigma$ be
$\Qq$-acyclic in codimension~one).  This has previously been done by
Lyons~\cite{Lyons} in terms of slightly different invariants.  If~$G$
is a connected graph, then $\tau(G)$ is just the number of spanning
trees, and we recover the classical matrix-tree theorem of Kirchhoff.

\subsection{Lattices} \label{lattice-subsection}

Starting in Section~\ref{integral-bases-section},
we will turn our attention to lattices
of integer cuts and flows.   We review
some of the general theory of integer lattices; see, e.g.,
\cite[Chapter~12]{Artin}, 
\cite[Chapter 14]{GodRoy},
\cite[Chapter IV]{Hungerford}.

A \emph{lattice} $\LL$ is a discrete subgroup of a finite-dimensional
vector space~$V$; that is, it is the set of integer linear combinations of some
basis of~$V$.  Every lattice $\LL\subseteq\Rr^n$ is isomorphic to
$\Zz^r$ for some integer $r\leq n$, called the \emph{rank} of~$\LL$.
The elements of $\LL$ span a vector space denoted by $\LL\otimes\Rr$.  
For $\LL\subseteq\Zz^n$, the \emph{saturation} of $\LL$ is defined 
as $\hat\LL=(\LL\otimes\Rr)\cap\Zz^n$.
An \emph{integral basis} of $\LL$ is a set of linearly independent
vectors $v_1,\dots,v_r\in\LL$ such that $\LL=\{c_1v_1+\cdots+c_rv_r
\st c_i\in\Zz\}$.
We will need the following fact about integral bases of lattices;
the equivalences are easy consequences of the theory of free modules
(see, e.g., \cite[Chapter~12]{Artin}, \cite[Chapter IV]{Hungerford}):
\begin{proposition}\label{TFAE-lattices}
For any lattice $\LL\subseteq\Zz^n$, the following are equivalent:
\begin{enumerate}
\item Every integral basis of $\LL$ can be extended to an integral basis of $\Zz^n$.
\item Some integral basis of $\LL$ can be extended to an integral basis of $\Zz^n$.
\item $\LL$ is a summand of $\Zz^n$, i.e., $\Zz^n$ can be written as an internal direct sum $\LL\oplus\LL'$.
\item $\LL$ is the kernel of some group homomorphism $\Zz^n\to\Zz^m$.
\item $\LL$ is saturated, i.e., $\LL=\hat\LL$.
\item $\Zz^n/\LL$ is a free $\Zz$-module, i.e., its torsion submodule is zero.
\end{enumerate}
\end{proposition}

Fixing the standard inner product $\langle\cdot,\cdot\rangle$ on $\Rr^n$,
we define the \emph{dual lattice} of $\LL$ by
$$\LL^\sharp = \{v\in\LL\otimes\Rr \st \scp{v,w}\in\Zz \ \ \forall
w\in\LL\}.$$
Note that $\LL^\sharp$ can be identified with the dual
$\Zz$-module $\LL^*=\Hom(\LL,\Zz)$, and that $(\LL^\sharp)^\sharp =
\LL$.  A lattice is called \emph{integral} if it is contained in its
dual; for instance, any subgroup of $\Zz^n$ is an integral lattice.
The \emph{discriminant group} (or \emph{determinantal group})
of an integral lattice $\LL$ is $\LL^\sharp / \LL$; its cardinality
can be calculated as $\det M^TM$,
for any matrix~$M$ whose columns form an integral basis of $\LL$.
We will need the following facts about bases and duals of lattices.
\begin{proposition}\label{lattice-projection}\cite[Section~14.6]{GodRoy}
Let $M$ be an $n\times r$ integer matrix.
\begin{enumerate}
\item If the columns of $M$ form an integral basis for the lattice $\LL$, 
then the columns of $M(M^TM)^{-1}$ form the corresponding dual basis for $\LL^\sharp$.
\item The matrix $P=M(M^TM)^{-1}M^T$ represents orthogonal projection from $\Rr^n$ 
onto the column space of $M$.
\item If the greatest common divisor of the 
$r \times r$ minors of $M$ is 1, then $\LL^\sharp$ is generated by the columns of~$P$.
\end{enumerate}
\end{proposition}

\subsection{The cellular matroid}  \label{matroid-subsection}
Many ideas of the paper may be expressed efficiently using
the language of matroids.  For a general reference on matroids, see, e.g., \cite{Oxley}.
We will primarily consider cellular matroids. The
\emph{cellular matroid} of~$\Sigma$ is the matroid $\MM(\Sigma)$ represented
over~$\Rr$ by the columns of the boundary matrix~$\bd$.  Thus the ground set of $\MM(\Sigma)$
naturally corresponds to the $d$-dimensional cells~$\Sigma_d$, and the matroid records which sets of
columns of~$\bd$ are linearly independent.  If~$\Sigma$ is a graph, then
$\MM(\Sigma)$ is its usual graphic matroid, while if $\Sigma$ is a simplicial
complex then $\MM(\Sigma)$ is its simplicial matroid (see~\cite{CorLin}).

The bases of $\MM(\Sigma)$ are the collections of facets of cellular
spanning forests of $\Sigma$. If $r$ is the rank function of the
matroid $\MM(\Sigma)$, then for each set of facets $B\subseteq \Sigma_d$,
we have $r(B) =
\rank \bd_B$, where $\bd_B$ is the submatrix consisting of the columns
indexed by the facets in~$B$. Moreover, we have
$$r(\Sigma) := r(\Sigma_d) = \rank \MM(\Sigma) = \rank \bd = |\Sigma_d|-\betti_d(\Sigma)$$
by the definition of Betti number.

A set of facets $B\subseteq \Sigma_d$ is called a \emph{cut} if deleting
$B$ from $\Sigma$ increases its codimension-one homology, i.e.,
$\betti_{d-1}(\Sigma\sm B)>\betti_{d-1}(\Sigma)$.  A cut $B$ is a \emph{bond} if
$r(\Sigma \sm B) = r(\Sigma)-1$, but $r((\Sigma \sm B) \cup\sigma)=r(\Sigma)$ for every
$\sigma\in B$.  That is, a bond is a minimal cut.  In matroid terminology, a bond
of $\Sigma$ is precisely a cocircuit of $\MM(\Sigma)$, i.e., a minimal set that
meets every basis of $\MM(\Sigma)$.  Equivalently, a bond is the complement
of a flat of rank $r(\Sigma)-1$.  If $\Upsilon$ is a cellular spanning forest (i.e., a basis of $\MM(\Sigma)$)
and $\sigma\in\Upsilon_d$ is a facet, then the \emph{fundamental bond}
of the pair $(\Upsilon,\sigma)$ is
\begin{equation} \label{fund-bond}
\bond(\Upsilon,\sigma) = \sigma \cup \left\{\rho\in \Sigma_d\sm\Upsilon\st \Upsilon\sm\sigma\cup\rho\text{ is a CSF}\right\}.
\end{equation}
This is the fundamental cocircuit of the pair $(\Upsilon,\sigma)$ of $\MM(\Sigma)$
\cite[p.~78]{Oxley}.

While the language of matroids will frequently be useful, it is
important to point out that most of the objects of interest to us, such as the
cut and flow lattices and the critical group of a cell complex~$\Sigma$, are \emph{not} purely combinatorial
invariants of its cellular matroid $\MM(\Sigma)$.  (See \cite{SW} for more on
this subject, and \cite{Moci1,Moci2} for generalizations of matroids that contain
finer arithmetic information).
As an example, the summands in \eqref{complexity} are indexed by the bases
of $\MM(\Sigma)$, but the summands themselves are not part of
the matroid data.  (On the other hand, when $\Sigma$ is a graph, all summands are 1.)

Below is a table collecting some of the
standard terminology from linear algebra, graph theory, and matroid
theory, along with the analogous concepts that we will be using for
cell complexes.

\begin{center}
\begin{scriptsize}
\begin{tabular}{|cccc|}\hline
\bf Linear algebra & \bf Graph & \bf Matroid & \bf Cell complex\\ \hline
Column vectors & Edges & Ground set & Facets\\
Independent set & Acyclic subgraph & Independent set & Acyclic subcomplex\\
Min linear dependence & Cycle & Circuit & Circuit\\
Basis & Spanning forest & Basis & CSF\\
Set meeting all bases & Disconnecting set & Codependent
set & Cut\\
Min set meeting all bases & Bond & Cocircuit & Bond\\
Rank & \# edges in spanning forest & Rank & \# facets in CSF\\\hline
\end{tabular}
\end{scriptsize}
\end{center}

Here ``codependent'' means dependent in the dual matroid.

\section{Enumerating Cellular Spanning Forests}\label{cellular-forest}

In this section, we study the enumerative properties of cellular spanning forests
of an arbitrary cell complex $\Sigma$.  Our setup is essentially the same as that
of Lyons \cite[\S6]{Lyons}, but the combinatorial formulas we will need later, namely Propositions~\ref{det-is-homology} and~\ref{relative-tor-formula}, are somewhat different.
As a corollary, we obtain an enumerative result, Proposition~\ref{CMFT}, which generalizes the simplicial and cellular matrix-tree theorems
of~\cite{Simplicial} and~\cite{Cellular} (in which we required that $\Sigma$ be $\Qq$-acyclic
in codimension~one).  The result is closely related, but not quite equivalent, to Lyons' generalization of the cellular matrix-tree theorem \cite[Corollary 6.2]{Lyons}, and to \cite[Corollary D]{Catanzaro}.

The arguments require some tools from homological algebra,
in particular the long exact sequence for relative homology and some facts
about the torsion-subgroup functor.  The details of the proofs are not
necessary to understand the constructions of cut and flow spaces in the
later sections.

Let $\Sigma$ be a $d$-dimensional cell complex with rank~$r$.
Let $\Gamma\subseteq\Sigma$ be a 
subcomplex of dimension less than or equal to $d-1$ such that  $\Gamma_{(d-2)}=\Sigma_{(d-2)}$.
Thus the inclusion map $i\colon \Gamma\to \Sigma$ induces isomorphisms 
$i_*\colon \HH_k(\Gamma;\Qq)\to\HH_k(\Sigma;\Qq)$ for all $k<d-2$.

\begin{definition} \label{rel-acyclic}
The subcomplex $\Gamma\subseteq\Sigma$ is called \emph{relatively acyclic} if in fact
the inclusion map 
$i\colon \Gamma\to \Sigma$ induces isomorphisms 
$i_*\colon \HH_k(\Gamma;\Qq)\to\HH_k(\Sigma;\Qq)$ for all $k<d.$
\end{definition}

By the long exact sequence for relative homology, $\Gamma$ is relatively acyclic if and only if
$\HH_d(\Sigma;\Qq)\to\HH_d(\Sigma,\Gamma;\Qq)$ is an isomorphism and $\HH_k(\Sigma,\Gamma;\Qq)=0$ for all $k<d$.
These conditions can occur only if $|\Gamma_{d-1}|=|\Sigma_{d-1}|-r$.  This quantity may be zero
(in which case the only relatively acyclic subcomplex is $\Sigma_{(d-2)}$).  A relatively acyclic subcomplex is precisely the complement of a $(d-1)$-cobase (a basis of the matroid represented over~$\Rr$ by the rows of the boundary matrix~$\bd$)
in the terminology of Lyons \cite{Lyons}.

Two special cases are worth noting.  First, if $d=1$, then a relatively acyclic
complex consists of one vertex in each connected component.
Second, if $\HH_{d-1}(\Sigma;\Qq)=0$, then $\Gamma$ is relatively acyclic if and only if
it is a cellular spanning forest of $\Sigma_{(d-1)}$.

For a matrix $M$, we write $M_{A,B}$ for the restriction of
$M$ to rows indexed by $A$ and columns indexed by $B$. 

\begin{proposition} \label{det-is-homology}
Let $\Gamma\subseteq\Upsilon\subseteq\Sigma$ be subcomplexes
such that $\dim\Upsilon=d$; $\dim\Gamma=d-1$; $|\Upsilon_d|=r$;
$|\Gamma_{d-1}|=|\Sigma_{d-1}|-r$;
$\Upsilon_{(d-1)}=\Sigma_{(d-1)}$; and $\Gamma_{(d-2)}=\Sigma_{(d-2)}$.
Also, let $R=\Sigma_{d-1} \sm \Gamma$.  Then the following are equivalent:
\begin{enumerate}
\item The $r\x r$ square matrix $\hat\bd=\bd_{R,\Upsilon}$ is nonsingular.
\item $\HH_d(\Upsilon,\Gamma;\Qq)=0$.
\item $\HH_{d-1}(\Upsilon,\Gamma;\Qq)=0$.
\item $\Upsilon$ is a cellular spanning forest of $\Sigma$ and $\Gamma$ is relatively acyclic.
\end{enumerate}
\end{proposition}

\begin{proof}
The cellular chain complex of the relative complex $(\Upsilon,\Gamma)$ is
$$0 \to C_d(\Upsilon,\Gamma;\Qq)=\Qq^r \xrightarrow{\hat\bd} C_{d-1}(\Upsilon,\Gamma;\Qq)=\Qq^r \to 0$$
with other terms zero.  If $\hat\bd$ is nonsingular, then $\HH_d(\Upsilon,\Gamma;\Qq)$ and $\HH_{d-1}(\Upsilon,\Gamma;\Qq)$ are both zero;
otherwise, both are nonzero.  This proves the equivalence of (a), (b) and (c).

Next, note that $\HH_d(\Gamma;\Qq)=0$ (because $\Gamma$ has no cells in dimension~$d$) and that
$\HH_{d-2}(\Upsilon,\Gamma;\Qq)=0$ (because $\Gamma_{(d-2)}=\Upsilon_{(d-2)}$). 
Accordingly, the long exact sequence for relative homology of $(\Upsilon,\Gamma)$ is
\begin{equation} \label{LES-relative}
\begin{aligned}
0 &\to \HH_d(\Upsilon;\Qq) \to \HH_d(\Upsilon,\Gamma;\Qq)\\
 &\to \HH_{d-1}(\Gamma;\Qq) \to \HH_{d-1}(\Upsilon;\Qq) \to \HH_{d-1}(\Upsilon,\Gamma;\Qq)\\
 &\to \HH_{d-2}(\Gamma;\Qq) \to \HH_{d-2}(\Upsilon;\Qq) \to 0.
\end{aligned}
\end{equation}
If $\HH_d(\Upsilon,\Gamma;\Qq)=\HH_{d-1}(\Upsilon,\Gamma;\Qq)=0$, then
$\HH_d(\Upsilon;\Qq)=0$ (which says that~$\Upsilon$ is a cellular spanning forest) and the rest of~\eqref{LES-relative}
splits into two isomorphisms that assert precisely that~$\Gamma$ is relatively acyclic (recall that $\HH_{d-1}(\Upsilon;\Qq)=\HH_{d-1}(\Sigma;\Qq)$
when~$\Upsilon$ is a cellular spanning forest).  This implication is reversible,
completing the proof.
\end{proof}

The \emph{torsion subgroup} of a finitely generated abelian group $A$ is
defined as the subgroup
$$\tor(A)=\{x\in A \st kx=0\text{ for some } k\in\Zz\}.$$
Note that $A=\tor(A)$ if and only if $A$ is finite.  The torsion functor
$\tor$
is left-exact~\cite[ p.~179]{Hungerford}.
Moreover, if $A\to B\to C\to 0$ is exact and $A=\tor(A)$, then
$\tor(A) \to \tor(B) \to \tor(C) \to 0$ is exact.  We will need the following additional fact about the torsion functor.

\begin{lemma} \label{exactness-of-tor}
Suppose we have a commutative diagram of finitely generated abelian groups
\begin{equation} \label{first-comm-diag}
\xymatrix{
0\rto& A\rto^f\dto_\alpha& B\rto^g\dto^\beta& C\rto^h\dto& D\rto^j\dto& E\rto\dto& 0\\
0\rto& A'\rto^{f'}& B'\rto^{g'}& C'\rto^{h'}& D'\rto^{j'}& E'\rto& 0
}
\end{equation}
such that both rows are exact;  $A,A'$ are free; $\alpha$ is an isomorphism; $\beta$ is surjective; and $C,C'$ are finite.
Then there is an induced commutative diagram
\begin{equation} \label{induced-comm-diag}
\xymatrix @C=0.2in @R=0.2in{
0\rto& \tor B\oplus G\rto\dto&\tor C\rto\dto& \tor D\rto\dto& \tor E\rto\dto& 0\\
0\rto& \tor B'\oplus G\rto& \tor C'\rto& \tor D'\rto& \tor E'\rto& 0
}
\end{equation}
such that $G$ is finite and both rows are exact.  Consequently
\begin{equation} \label{card-formula}
|\tor B|\cdot |\tor C'|\cdot |\tor D| \cdot |\tor E'|=|\tor B'|\cdot |\tor C|\cdot|\tor D'|\cdot |\tor E|.
\end{equation}
\end{lemma}

\begin{proof}
%
Since $C$ is finite, we have $\ker j=\im h\subseteq\tor D$,
so replacing $D,E$ with their torsion summands
preserves exactness.  The same argument implies that we can replace $D',E'$ with $\tor D',\tor E'$.

Second, note that $A,A',B,B'$ all have the same rank (since the rows are exact, $C,C'$ are finite,
and $\alpha$ is an isomorphism).  Hence $f(A)$ is a maximal-rank
free submodule of $B$; we can write $B=\tor B\oplus F$, where
$F$ is a free summand of $B$ containing $f(A)$.  Likewise,
write $B'=\tor B'\oplus F'$, where
$F'$ is a free summand of $B'$ containing $f'(A')$.  Meanwhile, $\beta$ is surjective, hence
must restrict to an isomorphism $F\to F'$,
which induces an isomorphism $F/f(A)\to F'/f'(A')$.  Abbreviating this last group by $G$, we obtain
the desired diagram~\eqref{induced-comm-diag} .  Since $\ker g=\im f\subseteq F$, the map $g\colon\tor B\oplus G\to \tor C$
is injective, proving exactness of the first row; the second row is exact by the same argument.
Exactness of each row implies that the alternating product of the cardinalities of the groups is 1,
from which the formula~\eqref{card-formula} follows.
\end{proof}

\begin{proposition} \label{relative-tor-formula}
Let $\Sigma$ be a $d$-dimensional cell complex,
let $\Upsilon\subseteq\Sigma$ be a cellular spanning forest, and let $\Gamma\subseteq\Sigma$ be a relatively acyclic $(d-1)$-subcomplex.
Then
$$\torh_{d-1}(\Upsilon)\, \torh_{d-1}(\Sigma,\Gamma)=\torh_{d-1}(\Sigma)\, \torh_{d-1}(\Upsilon,\Gamma).$$
\end{proposition}

\begin{proof}
The inclusion $\Upsilon\subseteq\Sigma$ induces a commutative diagram
$$
\xymatrix@C=0.15in{
0 \rto& \HH_{d-1}(\Gamma;\Zz) \rto^{i_*}\dto^{\isom}& \HH_{d-1}(\Upsilon;\Zz)\rto^{j_*}\dto& \HH_{d-1}(\Upsilon,\Gamma;\Zz) \rto\dto& \HH_{d-2}(\Gamma;\Zz) \rto\dto^{\isom}& \HH_{d-2}(\Upsilon;\Zz) \rto\dto^{\isom}& 0\\
0 \rto& \HH_{d-1}(\Gamma;\Zz) \rto^{i_*}& \HH_{d-1}(\Sigma;\Zz) \rto^{j_*} & \HH_{d-1}(\Sigma,\Gamma;\Zz) \rto& \HH_{d-2}(\Gamma;\Zz) \rto& \HH_{d-2}(\Sigma;\Zz) \rto& 0
}$$
whose rows come from the long exact sequences
for relative homology.  (For the top row,
the group $\HH_d(\Upsilon,\Gamma;\Zz)$ is free because $\dim \Upsilon = d$,
and on the other hand is purely torsion by Proposition~\ref{det-is-homology},
so it must be zero.  For the bottom row, 
the condition that $\Gamma$ is relatively
acyclic implies that $i_*$ is an isomorphism over $\Qq$; therefore,
it is one-to-one over $\Zz$.)
The groups $\HH_{d-1}(\Upsilon,\Gamma;\Zz)$ and
$\HH_{d-1}(\Sigma,\Gamma;\Zz)$ are purely torsion.
The first, fourth and fifth vertical maps are isomorphisms (the last
because $\Upsilon_{(d-1)}=\Sigma_{(d-1)}$) and the second is a surjection by the relative homology sequence of
the pair $(\Sigma,\Upsilon)$ (since the relative complex has no cells in dimension $d-1$).
The result now follows by applying Lemma~\ref{exactness-of-tor} and canceling like terms.
\end{proof}

As a consequence, we obtain a version of the cellular matrix-forest
theorem that applies to all cell complexes (not only those that are
$\Qq$-acyclic in codimension~one).

\begin{proposition} \label{CMFT}
Let $\Sigma$ be a $d$-dimensional cell complex and
let $\Gamma\subseteq\Sigma$ be a relatively acyclic $(d-1)$-dimensional subcomplex, and let
$L_\Gamma$ be the restriction of $\Lud_{d-1}(\Sigma)$ to the $(d-1)$-cells of~$\Gamma$.  Then
$$\tau_d(\Sigma) = \frac{\torh_{d-1}(\Sigma)^2}{\torh_{d-1}(\Sigma,\Gamma)^2}\det L_\Gamma.$$
\end{proposition}

\begin{proof}
By the Binet-Cauchy formula and Propositions~\ref{det-is-homology}
and~\ref{relative-tor-formula}, we have
\begin{align*}
\det L_\Gamma &= \det \bd_\Gamma\cbd_\Gamma = \sum_{\Upsilon\subseteq \Sigma_d\st |\Upsilon|=r(\Sigma)} (\det\bd_{\Gamma,\Upsilon})^2
~=~ \sum_{\text{CSFs\ }\Upsilon\subseteq \Sigma_d} \torh_{d-1}(\Upsilon,\Gamma)^2\\
&= \frac{\torh_{d-1}(\Sigma,\Gamma)^2}{\torh_{d-1}(\Sigma)^2}\sum_{\text{CSFs\ }\Upsilon\subseteq \Sigma_d} \torh_{d-1}(\Upsilon)^2
~=~ \frac{\torh_{d-1}(\Sigma,\Gamma)^2}{\torh_{d-1}(\Sigma)^2}\tau_d(\Sigma)
\end{align*}
and solving for $\tau_d(\Sigma)$ gives the desired formula.
\end{proof}
If $\HH_{d-1}(\Sigma;\Zz)=\tor(\HH_{d-1}(\Sigma;\Zz))$, then the
relative homology sequence of the pair $(\Sigma,\Gamma)$ gives rise to the exact sequence
$$0\to\tor(\HH_{d-1}(\Sigma;\Zz))\to\tor(\HH_{d-1}(\Sigma,\Gamma;\Zz))\to\tor(\HH_{d-2}(\Gamma;\Zz))\to\tor(\HH_{d-2}(\Sigma;\Zz))\to0$$
which implies that
$\torh_{d-1}(\Sigma)/\torh_{d-1}(\Sigma,\Gamma)=\torh_{d-2}(\Sigma)/\torh_{d-2}(\Gamma)$, so
Proposition~\ref{CMFT} becomes the formula
$\tau_d(\Sigma) = \frac{\torh_{d-2}(\Sigma)^2}{\torh_{d-2}(\Gamma)^2}\det L_\Gamma$.  This was one of the original versions of the
cellular matrix-tree theorem~\cite[Theorem~2.8(2)]{Cellular}.

\begin{remark} \label{lyons-remark}
Lyons \cite[Corollary 6.2]{Lyons} proves a similar matrix-forest theorem in terms of an invariant $\torh'$ defined below.  He shows that each row of~\eqref{first-comm-diag} induces the corresponding row of~\eqref{induced-comm-diag}.  This does not quite imply Lemma~\ref{exactness-of-tor}, since one still needs to identify the ``error terms'' $G$ in the top and bottom rows of \eqref{induced-comm-diag}.  Doing so would amount to showing that $\HH_{d-1}(\Upsilon)/\ker(j_*)\isom\HH_{d-1}(\Sigma)/\ker(j_*)$ 
in the commutative diagram of Proposition~\ref{relative-tor-formula}.
 Alternatively, Proposition~\ref{relative-tor-formula} would follow from \cite[Lemma 6.1]{Lyons} together with the equation
$$\torh_{d-2}(\Gamma) \torh'_{d-1}(\bar\Gamma) / \torh_{d-2}(\Sigma) = \torh_{d-1}(\Sigma,\Gamma) / \torh_{d-1}(\Sigma)$$
where $\bar\Gamma=X_{(d-1)}\sm\Gamma$ and $\torh'_{d-1}(\bar\Gamma)=
\big\vert \ker\bd_{d-1}(\Sigma;\Zz) / 
  \big( (\ker\bd_{d-1}(\Sigma;\Zz)\cap\im\bd_d(\Sigma;\Qq)) + \ker\bd_{d-1}(\Gamma;\Zz) \big)
\big\vert$.
\end{remark}

\section{The Cut Space}\label{cutspace-section}

Throughout this section, let $\Sigma$ be a cell complex of dimension~$d$ and rank~$r$
(that is, every cellular spanning forest of~$\Sigma$ has $r$ facets).
For each~$i\leq d$, the \textit{$i$-cut space} and \textit{$i$-flow space} of $\Sigma$ are defined
respectively as the spaces of cellular coboundaries and cellular cycles:
\begin{align*}
\Cut_i(\Sigma) &= \im( \cbd_i\colon C_{i-1}(\Sigma,\Rr) \to C_i(\Sigma,\Rr)),\\
\Flow_i(\Sigma) &= \ker( \bd_i\colon C_i(\Sigma,\Rr) \to C_{i-1}(\Sigma,\Rr)).
\end{align*}
We will primarily be concerned with the case $i=d$.
For $i=1$, these are the standard graph-theoretic cut and
flow spaces of the 1-skeleton of $\Sigma$.

There are two natural ways to construct bases of the cut space of a graph, 
in which the basis elements correspond to either (a) vertex stars or (b)
the fundamental circuits of a spanning forest (see, e.g.~\cite[Chapter~14]{GodRoy}).
The former is easy to
generalize to cell complexes, but the latter takes more work.

First, if $G$ is a graph on vertex set $V$ and $R$ is a set of (``root'')
vertices, one in each connected component, then the rows of $\bd$ corresponding
to the vertices $V\sm R$ form a basis for $\Cut_1(G)$.  This
observation generalizes easily to cell complexes:

\begin{proposition} \label{star-basis}
A set of $r$ rows of~$\bd$ forms a
row basis if and only if the corresponding set of $(d-1)$-cells is the
complement of a relatively acyclic $(d-1)$-subcomplex.
\end{proposition}

This is immediate from Proposition~\ref{det-is-homology}.  Recall that
if $\HH_{d-1}(\Sigma;\Qq)=0$, then ``relatively acyclic
$(d-1)$-subcomplex'' is synonymous with ``spanning tree of the
$(d-1)$-skeleton''.  In this case, Proposition~\ref{star-basis} is
also a consequence of the fact that the matroid represented by the
rows of~$\bd_d$ is dual to the matroid represented by the columns
of~$\bd_{d-1}$~\cite[Proposition~6.1]{Cellular}.

The second way to construct a basis of the cut space of a graph is to
fix a spanning tree and take the signed characteristic vectors of its
fundamental bonds.  In the cellular setting, it is not hard to show
that each bond supports a unique (up to scaling) vector in the cut
space (Lemma~\ref{one-dim}) and that the fundamental bonds of a fixed
cellular spanning forest give rise to a vector space basis
(Theorem~\ref{cut-basis}).  (Recall from Section~\ref{prelims} that a
bond in a cell complex is a minimal collection of facets whose removal
increases the codimension-one homology, or equivalently a cocircuit of
the cellular matroid.)  The hard part is to identify the entries of
these cut-vectors.  For a graph, these entries are all 0 or $\pm1$.
In higher dimension, this need not be the case, but the entries can be
interpreted as the torsion coefficients of certain subcomplexes
(Theorem~\ref{calibration}).  In Section~\ref{flowspace-section}, we
will prove analogous results for the flow space.

\begin{remark} \label{puncturing}
Although many of our results may be stated in terms of algebraic chain complexes over $\Zz$ (integer boundary matrices), we use the language of cell complexes.  (This is a difference only in terminology, not the generality of the results, since every integer matrix is the top-dimensional boundary matrix of some cell complex.)  Thus, definitions and results about column bases, row bases, rank, etc., can be interpreted topologically in terms of cellular spanning trees and creating and puncturing holes in cell complexes (see Example~\ref{bipyramid-cut} and Figure~\ref{BipyramidFigure}).  This is analogous to the situation in algebraic graph theory, where results that can be stated in terms of matrices are often more significant in terms of trees, cuts, flows, etc.
\end{remark}

\subsection{A basis of cut-vectors}

Recall that the \emph{support} of a vector $v=(v_1,\dots,v_n)\in\Rr^n$ is the set
$$\supp(v)=\{i\in[n] \st v_i\neq 0\}.$$

\begin{proposition}\label{bond-is-minimal} \cite[Proposition 9.2.4]{Oxley}.
Let $M$ be a $r\x n$ matrix with rowspace $V\subseteq\Rr^n$, and
let $\MM$ be the matroid represented by the columns of $M$.  Then the
 cocircuits of $\MM$ are the inclusion-minimal elements of the
family $\Supp(V) := \{\supp(v) \st v\in V\sm\{0\}\}$.
\end{proposition}

\begin{lemma}\label{one-dim}
Let $B$ be a bond of~$\Sigma$.  Then the set
$$\Cut_B(\Sigma)=\{0\}\cup\{v \in \Cut_d(\Sigma) \st \supp(v) = B\}$$
is a one-dimensional subspace of $\Cut_d(\Sigma)$.
That is, up to scalar multiple, there is a unique cut-vector
whose support is exactly $B$.  

\end{lemma}

\begin{proof}
Suppose that $v,w$ are vectors in the cut space, both supported on
$B$, that are not scalar multiples of each other.  Then there is a
linear combination of $v,w$ with strictly smaller support; this
contradicts Proposition~\ref{bond-is-minimal}.  On the other hand,
Proposition~\ref{bond-is-minimal} also implies that $\Cut_B(\Sigma)$ is not the zero
space; therefore, it has dimension~1.
\end{proof}

We now know that for every bond~$B$, there is a cut-vector supported
on~$B$ that is uniquely determined up to a scalar multiple.  As we
will see, there is a choice of scale so that the coefficients of this
cut-vector are given by certain minors of the down-up Laplacian
$L=\Ldu_d(\Sigma)=\cbd\bd$ (Lemma~\ref{cut-vector-formula}); these minors (up to sign) can be interpreted
as the cardinalities of torsion homology groups (Theorem~\ref{calibration}).

In choosing a scale, the first step is to realize the elements of $\Cut_B(\Sigma)$ explicitly as images
of the map $\cbd$.   Fix an inner product
$\scp{\cdot,\cdot}$ on each chain group $C_i(\Sigma;\Rr)$ by declaring
the $i$-dimensional cells to be an orthonormal basis.  (This amounts
to identifying each cell with the cochain that is its characteristic
function.)  Thus, for $\alpha\in C_i(\Sigma;\Rr)$, we have
$\supp(\alpha)=\{\sigma\in\Sigma_i \st \scp{\sigma,\alpha}\neq 0\}$.
Moreover, for all $\beta\in C_{i-1}(\Sigma;\Rr)$, we have by basic linear algebra
\begin{equation} \label{scalar-product-identity}
\scp{\bd\alpha,\beta}=\scp{\alpha,\cbd\beta}.
\end{equation}

\begin{lemma} \label{cbd-of-what}
Let $B$ be a bond of $\Sigma$ and let $U$ be the space spanned by $\{\bd \sigma \st \sigma\in \Sigma_d\sm B\}$.
In particular, $U$ is a subspace of $\im\bd$ of codimension one.
Let $V$ be the orthogonal complement of $U$ in $\im\bd$, and let
$v$ be a nonzero element of $V$.  Then $\supp(\cbd v)=B$.
\end{lemma}

\begin{proof}
First, we show that $\cbd v\neq 0$.  To see this, observe that
the column space of~$\bd$ is $U + \Rr v$, so
the column space of~$\cbd\bd$ is $\cbd U + \Rr \cbd v$.  However,
$\rank(\cbd\bd)=\rank\bd=r$, and $\dim U=r-1$;
therefore, $\cbd v$ cannot be the zero vector.  Second, if
$\sigma \in \Sigma_d \sm B$, then $\bd \sigma \in U$, so
$\scp{\cbd v,\sigma} = \scp{v,\bd \sigma} = 0.$
It follows that $\supp(\cbd v)\subseteq B$, and in fact
$\supp(\cbd v)=B$ by Proposition~\ref{bond-is-minimal}.
\end{proof}

Given a bond $B$, let $A=\{\sigma_1,\dots,\sigma_{r-1}\}$ be a
cellular spanning forest of~$\Sigma\sm B$.  Fix a facet
$\sigma=\sigma_r\in B$, so that $A\cup\sigma$ is a cellular spanning
forest of~$\Sigma$.  Define a vector
$$v=v_{A,\sigma}
=\sum_{j=1}^r (-1)^j (\det \Ldu_{A,A\cup \sigma\sm \sigma_j}) \bd \sigma_j
~\in~ C_{d-1}(\Sigma;\Zz)$$
so that
\begin{equation} \label{cbdv}
\cbd v=\sum_{j=1}^r (-1)^j (\det \Ldu_{A,A\cup \sigma\sm \sigma_j}) \Ldu \sigma_j ~\in~ \Cut_d(\Sigma).
\end{equation}

\begin{lemma} \label{cut-vector-formula}
For the cut-vector $\cbd v$ defined in equation~\eqref{cbdv},
$$\cbd v=(-1)^r\sum_{\rho\in B} (\det \Ldu_{A\cup \rho,A\cup \sigma}) \rho.$$
In particular, $\supp(\cbd v)=B$.
\end{lemma}

\begin{proof}
For each $\rho\in B$, 
\begin{align*}
\scp{\cbd v,\rho}
 &= \sum_{j=1}^r (-1)^j \det \Ldu_{A,A\cup \sigma\sm \sigma_j} \scp{\Ldu \sigma_j,\rho}\\
 &= \sum_{j=1}^r (-1)^j \det \Ldu_{A,A\cup \sigma\sm \sigma_j} \scp{\bd \sigma_j,\bd \rho}\\
 &= \sum_{j=1}^r (-1)^j \det \Ldu_{A,A\cup \sigma\sm \sigma_j} \Ldu_{\rho,\sigma_j}\\
 &= (-1)^r\det \Ldu_{A\cup \rho,A\cup \sigma},
\end{align*}
where the last equality comes from expanding the row corresponding
to~$\rho$.  Note that $\det \Ldu_{A\cup \rho,A\cup \sigma}\neq0$
for
$\rho=\sigma$, so $\cbd v\neq 0$.  On the other hand, by Cramer's
rule, $v$ is orthogonal to $\bd \sigma_1,\dots, \bd \sigma_{r-1}$, so in fact $\scp{\cbd v,\rho}=0$ for all $\rho\in \Sigma_d \sm B$.
This establishes the desired formula for $\cbd v$, and then $\supp(\cbd v)=B$ by Lemma~\ref{cbd-of-what}.
\end{proof}

Equation~\eqref{cbdv} does not provide a canonical cut-vector
associated to a given bond~$B$, because $\cbd v$ depends on the choice of~$A$
and~$\sigma$.  On the other hand, the bond~$B$ can always be expressed as a fundamental
bond $\bond(\Upsilon,\sigma)$
(equivalently, fundamental cocircuit; see equation~\eqref{fund-bond} in Section~\ref{matroid-subsection}) by 
taking~$\sigma$ to be an arbitrary facet of~$B$ and taking
$\Upsilon=A\cup\sigma$, where~$A$ is a maximal acyclic subset of
$\Sigma\sm B$.  This observation suggests that the underlying combinatorial data
that gives rise to a cut-vector is really the pair $(\Upsilon,\sigma)$.

\begin{definition} \label{uncal-bond-vector}
Let $\Upsilon = \{\sigma_1, \sigma_2, \ldots, \sigma_r\}$ be a cellular spanning forest of~$\Sigma$, and let $\sigma=\sigma_i\in\Upsilon$.
The \emph{(uncalibrated) characteristic vector} of the bond $\bond(\Upsilon, \sigma)$ is:
$$\uncalcut(\Upsilon,\sigma) = (-1)^r \sum_{j=1}^r (-1)^j (\det \Ldu_{\Upsilon\sm\sigma, \Upsilon\sm\sigma_j}) \Ldu\sigma_j$$
By Lemma~\ref{cut-vector-formula}, taking $A=\Upsilon\sm \sigma$, we have
$$
\uncalcut(\Upsilon,\sigma) = \sum_{\rho\in\bond(\Upsilon,\sigma)} (\det \Ldu_{\Upsilon\sm \sigma\cup \rho,\Upsilon}) \rho,
$$
a cut-vector supported on $\bond(\Upsilon,\sigma)$.
\end{definition}

The next result is the cellular analogue of~\cite[Lemma~14.1.3]{GodRoy}.

\begin{theorem} \label{cut-basis}
The family $\{\uncalcut(\Upsilon,\sigma) \st \sigma\in \Upsilon\}$
is an $\Rr$-vector space basis for the cut space of $\Sigma$.
\end{theorem}

\begin{proof}
Let $\sigma\in\Upsilon$.  Then $\supp\uncalcut(\Upsilon,\sigma)=\bond(\Upsilon,\sigma)$ contains $\sigma$, but no other facet
of~$\Upsilon$.  Therefore, the set of characteristic vectors
is linearly independent,
and its cardinality is $|\Upsilon_d|=r=\dim\Cut_d(\Sigma)$.
\end{proof}

\begin{example}\label{bipyramid-cut}
The \emph{equatorial bipyramid} is the two-dimensional
simplicial complex $\Theta$ with facet set $\{123,124,125,134,135,234,235\}$
(Figure~\ref{BipyramidFigure}(a)).
Let $\Upsilon$ be the simplicial spanning tree with facets $\{123,124,234,135,235\}$ (unfolded in Figure~\ref{BipyramidFigure}(b)).  Then
\begin{align*}
\bond(\Upsilon,123)&=\{123,125,134\}, & \bond(\Upsilon,124)&=\{124,134\},\\
\bond(\Upsilon,135)&=\{135,125\}, & \bond(\Upsilon,234)&=\{234,134\},\\
\bond(\Upsilon,235)&=\{235,125\}.
\end{align*}
In each case, the removal of the bond leaves a 1-dimensional hole (as shown for the bond $\{123, 125, 134\}$ in Figure~\ref{BipyramidFigure}(c)).
By Theorem~\ref{cut-basis}, we have
\begin{align*}
\uncalcut(\Upsilon,123) &= 75([123]+[125]-[134]), &
\uncalcut(\Upsilon,124) &= 75([124]+[134]),\\
\uncalcut(\Upsilon,135) &= 75([125]+[135]), &
\uncalcut(\Upsilon,234) &= 75([134]+[234]),\\
\uncalcut(\Upsilon,235) &= 75([235]-[125]),
\end{align*}
which indeed form a basis for $\Cut_2(\Theta)$.
\begin{figure}[tb]
\includefigure{5.44in}{1.88in}{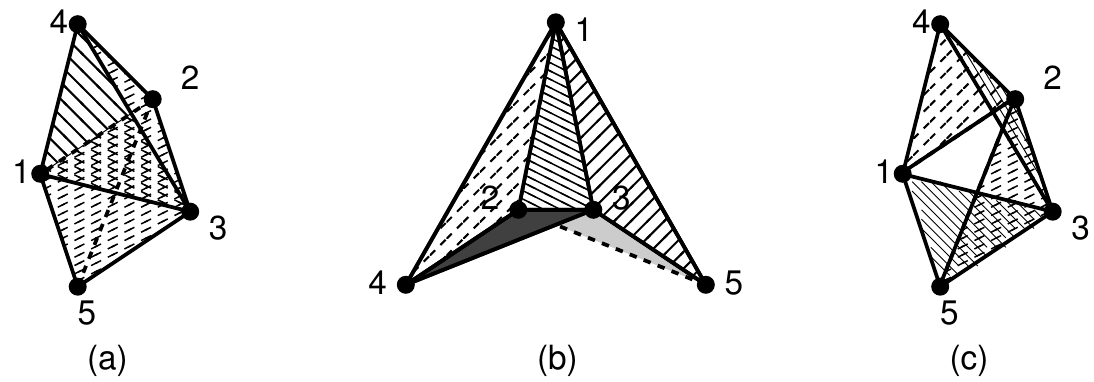}
\caption{\label{BipyramidFigure}
(a) The bipyramid $\Theta$. (b) A simplicial spanning tree
(unfolded). (c) Deleting the bond $\{123,134,125\}$.}
\end{figure}
\end{example}

\subsection{Calibrating the characteristic vector of a bond}

The term ``characteristic vector'' suggests that the coefficients of
$\uncalcut(\Upsilon,\sigma)$ should all be~0 or~1, 
but this is not necessarily possible, even by scaling, as Example~\ref{double-ravioli} below will show.
We would like to define the characteristic vector of a bond so that
it carries combinatorial or topological information,
avoiding extra factors such as the 
75 in Example~\ref{bipyramid-cut}.
We will show that the number
\begin{equation} \label{calfactor}
\mu_\Upsilon :=
\torh_{d-1}(\Upsilon)\sum_\Gamma \frac{\torh_{d-1}(\Sigma,\Gamma)^2}{\torh_{d-1}(\Sigma)^2},
\end{equation}
where the sum runs over all relatively acyclic $(d-1)$-subcomplexes $\Gamma\subseteq \Sigma$,
is an integer that divides every coefficient of the characteristic
vector.  Moreover, we will show that the entries of $\frac{1}{\mu_\Upsilon}
\uncalcut(\Upsilon,\sigma)$ are (up to sign) the torsion coefficients
of the cellular forests $\{\Upsilon\sm\sigma\cup\rho\}$ for $\rho\in\bond(\Upsilon,\sigma)$.

Let $\varepsilon^A_{\sigma,\sigma'}$ be the relative sign of $\bd \sigma,\bd \sigma'$ with respect to $\bd A$; that is, it is $+1$ or $-1$ according to whether
$\bd \sigma$ and $\bd \sigma'$ lie on the same or the opposite sides of the hyperplane in $\im\bd$ spanned by $\bd A$.  In the language of oriented matroids, this sign is simply a product of the entries corresponding to $\bd \sigma$ and $\bd \sigma'$ in one of the cocircuits corresponding to the hyperplane and determines the relative signs of a basis orientation on $A \cup \sigma$ and $A \cup \sigma'$~\cite[Section 3.5]{RedBook}.

\begin{proposition} \label{calculate-L}
Let $\Sigma$ be a cell complex of rank~$r$.
Let $\Upsilon=A\cup \sigma$ and $\Upsilon'=A\cup \sigma'$ be $d$-dimensional cellular spanning forests of $\Sigma$.
Let $\Ldu_{\Upsilon,\Upsilon'}$ be the restriction of the down-up Laplacian $\Ldu=\cbd\bd$ to the rows indexed by $\Upsilon$
and the columns indexed by $\Upsilon'$.
Then
$$\det \Ldu_{\Upsilon,\Upsilon'} = \det\cbd_\Upsilon\bd_{\Upsilon'} = \varepsilon \mu_\Upsilon \torh_{d-1}(\Upsilon')$$
where $\varepsilon=\varepsilon^A_{\sigma,\sigma'}$.
\end{proposition}

\begin{proof}
By the Binet-Cauchy formula, we have
\begin{equation} \label{detL}
\det \Ldu_{\Upsilon,\Upsilon'} = \sum_S (\det \cbd_{\Upsilon,S})(\det\bd_{S,\Upsilon'}) = \sum_S (\det \bd_{S,\Upsilon})(\det\bd_{S,\Upsilon'})
\end{equation}
where the sum runs over all sets $S\subseteq \Sigma_{d-1}$ with
$|S|=r$, and $\bd_{S,\Upsilon}$ is the corresponding $r\x r$ submatrix
of~$\bd$.  By Proposition~\ref{det-is-homology}, $\det\bd_{S,\Upsilon}$ is nonzero if and only if $\Gamma_S=\Sigma_{d-1}\sm S$ is relatively acyclic, and for those summands 
Proposition~\ref{relative-tor-formula} implies
$|\det\bd_{S,\Upsilon}|=\torh_{d-1}(\Upsilon,\Gamma_S)=\torh_{d-1}(\Upsilon)\torh_{d-1}(\Sigma,\Gamma_S)/\torh_{d-1}(\Sigma)$.

We claim that every summand in equation~\eqref{detL} has the same sign, namely~$\varepsilon$.  To see this, observe that since $\Upsilon,\Upsilon'$ are both column bases of $\bd$, there are unique scalars $\{c_\alpha:\ \alpha\in\Upsilon\}$ such that $\bd\sigma'=\sum_{\alpha\in\Upsilon} c_\alpha \bd\alpha$; in particular, $c_\sigma$ is nonzero and has the same sign as $\varepsilon$.  This equation holds upon restricting to any set of rows $S$.  By linearity of the determinant, $\det \bd_{S,\Upsilon'}=c_\sigma\det\bd_{S,\Upsilon}$ for each $S$, so the sign of every summand in \eqref{detL} is the same as that of $c_\sigma$.  Therefore \eqref{detL} gives
\begin{align*}
\det \Ldu_{\Upsilon,\Upsilon'}
&= \sum_S (\det \bd_{S,\Upsilon})(\det\bd_{S,\Upsilon'})\\
&= \sum_S \varepsilon \left(\frac{\torh_{d-1}(\Upsilon)  \torh_{d-1}(\Sigma,\Gamma_S)}{\torh_{d-1}(\Sigma)}\right) 
  \left(\frac{\torh_{d-1}(\Upsilon')  \torh_{d-1}(\Sigma,\Gamma_S)}{\torh_{d-1}(\Sigma)}\right)\\
&= \varepsilon \mu_\Upsilon \torh_{d-1}(\Upsilon').\qedhere
\end{align*}
\end{proof}

A corollary of the proof is that $\mu_\Upsilon$ is an integer, for the following reason.  The number $\torh_{d-1}(\Upsilon')$ is the gcd of the $r \x r$ minors of $\bd_{\Upsilon'}$, or equivalently the $r \x r$ minors of $\bd$ using columns $\Upsilon'$.  In other words, $\torh_{d-1}(\Upsilon')$ divides $\det\bd_{S,\Upsilon'}$ in every summand of~\eqref{detL}.  Therefore it divides $\det\Ldu_{\Upsilon,\Upsilon'}$, and $\mu_\Upsilon =\pm\det\Ldu_{\Upsilon,\Upsilon'}/\torh_{d-1}(\Upsilon')$ is an integer.

\begin{theorem}\label{calibration}
Let $B$ be a bond.  Fix a facet $\sigma\in B$ and a cellular spanning forest $A\subseteq \Sigma_d\sm B$,
so that in fact $B=\bond(A\cup \sigma,\sigma)$.  Define
the \emph{characteristic vector of $B$ with respect to $A$} as
$$\calcut{A}{B} := \frac{1}{\mu_\Upsilon}\uncalcut(A\cup \sigma,\sigma)
=\sum_{\rho\in B} \varepsilon^A_{\sigma,\rho} \torh_{d-1}(A\cup \rho) \,\, \rho.$$
Then $\calcut{A}{B}$ is in the cut space of $\Sigma$, and has integer coefficients.
Moreover, it depends on the choice of $\sigma$ only up to sign.
\end{theorem}

\begin{proof}
Apply the formula of Proposition~\ref{calculate-L} to the 
formula of Definition~\ref{uncal-bond-vector} for the characteristic vector, and
factor the integer $\mu_\Upsilon$ out of every coefficient.
Meanwhile, replacing $\sigma$ with a different facet $\sigma'\in B$ merely
multiplies all coefficients by
$\varepsilon^A_{\sigma',\rho}/\varepsilon^A_{\sigma,\rho}=\varepsilon^A_{\sigma,\sigma'}\in\{\pm1\}$.
\end{proof}

\begin{remark}
We are grateful to an anonymous referee for suggesting the following construction of the vector $\calcut{A}{B}$.  Let $M\in\Zz^{k\x n}$ be a matrix of rank $r$ whose first $r$ columns are linearly independent, and let $P$ be the generalized inverse of the submatrix of $M$ consisting of the first $r$ columns, so that $Q=PM$ is the reduced row-echelon form of $M$.  Thus $P^{-1}Q=M$.  By Cramer's rule, for each $j\in[n]$, the $j^{th}$ entry in the $r^{th}$ row of $Q$ is
\[q_{r,j}=\frac{\det(m_1, \dots, m_{r-1}, m_j)}{\det(m_1,\dots,m_{r-1},m_r)},\]
where $m_i$ denotes the $i^{th}$ column of $M$.

Now let $\Sigma,B,A,\sigma$ be as in Theorem~\ref{calibration}, let $r=\rank(\Sigma)$, and let $M$ be the top boundary matrix of $\Sigma$, with its first $r-1$ columns labeled by the facets of $A$ and the $r$th column labeled by $\sigma$.  The numerator of $q_{r,j}$ is
\[\begin{cases}
0 &\text{ for } j\not\in B,\\
\varepsilon^A_{\sigma,\rho_j}\torh_{d-1}(A\cup\rho_j) &\text{ for } j\in B,
\end{cases}\]
where $\rho_j$ denotes the facet corresponding to the $j$th column.
Therefore, clearing the denominators from the $r^{th}$ row of~$M$ produces the characteristic vector $\calcut{A}{B}$ of Theorem~\ref{calibration}.
\end{remark}

If $\Upsilon$ is a cellular spanning forest of $\Sigma$
and $\sigma\in\Upsilon_d$, then we define
the \emph{characteristic vector} of the pair $(\Upsilon,\sigma)$ by
\begin{equation} \label{calcut-upsilon-sigma}
\calcutta{\Upsilon}{\sigma} := \calcut{\Upsilon\sm\sigma}{\bond(\Upsilon,\sigma)} = \frac{1}{\mu_\Upsilon}\uncalcut(\Upsilon,\sigma).
\end{equation}

\begin{example} \label{ta-da}
Let $\Theta$ be the bipyramid of Example~\ref{bipyramid-cut}.
Every cellular spanning forest $\Upsilon\subseteq \Theta$ is torsion-free.
Moreover, the relatively acyclic subcomplexes $\Gamma$ that appear in equation~\eqref{calfactor}
are the spanning trees of the 1-skeleton $\Theta_{(1)}$ (see Section~\ref{cellular-forest}), which is the graph $K_5$ with one edge removed;
Accordingly, we have $\mu_\Upsilon=\tau(\Theta_{(1)})=75$, so the calibrated characteristic vectors
are as in Example~\ref{bipyramid-cut}, with all factors of 75 removed.
\end{example}

On the other hand, $\mu_\Upsilon$ is not necessarily the \emph{greatest} common factor of the
entries of each uncalibrated characteristic vector, as the following example illustrates.

\begin{example}\label{double-ravioli}
Consider the cell complex $\Sigma$ with a single vertex $v$, 
two 1-cells $e_1$ and $e_2$ attached at $v$, and four 2-cells
attached via the boundary matrix
$$
\bd=\bordermatrix{&\sigma_2&\sigma_3&\sigma_5&\sigma_7\cr
e_1&2&3&0&0\cr e_2&0&0&5&7}.
$$
Let $B$ be the bond $\{\sigma_2,\sigma_3\}$, so that the obvious candidate for
a cut-vector supported on~$B$ is the row vector $\begin{bmatrix}2&3&0&0\end{bmatrix}$.
On the other hand, taking $A=\{\sigma_5\}$ (a cellular spanning forest of $\Sigma \sm B$),
the calibrated characteristic vector given by Theorem~\ref{calibration} is
$$\calcut{A}{B}=\begin{bmatrix}10&15&0&0\end{bmatrix}.$$
For $\Upsilon=A\cup\{\sigma_2\}$, the uncalibrated characteristic vector of Definition~\ref{uncal-bond-vector} is
$$\uncalcut(\Upsilon,\sigma_2)=\begin{bmatrix}100&150&0&0\end{bmatrix}.$$
On the other hand, the calibration
factor $\mu_\Upsilon$ is not $\gcd(100,150)=50$, but rather 10, since $\torh_1(\Upsilon)=10$ and the summation of equation~\eqref{calfactor} has only one term,
namely $\Gamma=\Sigma_{(0)}$.  Similarly, for $A'=\{\sigma_7\}$ and $\Upsilon'=A'\cup\{\sigma_2\}$, we have
$$\calcut{A'}{B}=\begin{bmatrix}14&21&0&0\end{bmatrix},\quad
\uncalcut(\Upsilon',\sigma_2)=\begin{bmatrix}196&294&0&0\end{bmatrix},\quad
\mu_{\Upsilon'}=14.$$
\end{example}

\begin{remark} \label{torsion-in-simplicial-case}
As an illustration of where torsion plays a role, and of the principle that the cellular matroid $\MM(\Sigma)$ does not
provide complete information about cut-vectors, let $\Sigma$ be the
the complete $2$-dimensional simplicial complex on $6$ vertices, which has
complexity $6^6=46656$ \cite[Theorem 1]{Kalai}.  Most of
the cellular spanning trees of $\Sigma$ are contractible topological spaces,
hence $\Zz$-acyclic, and the calibrated cut-vectors obtained from them have
all entries equal to 0 or $\pm1$.  On the other hand, $\Sigma$ has twelve spanning trees $\Upsilon$ homeomorphic to the real projective plane (so that $\HH_1(\Upsilon;\Zz)\isom\Zz_2$).  For any facet $\sigma\in\Upsilon$, we have
$\bond(\Upsilon,\sigma)=\Sigma_2\sm\Upsilon_2\cup\{\sigma\}$,
and the calibrated cut-vector contains a $\pm2$ in position $\sigma$
and $\pm1$'s in positions $\Sigma\sm\Upsilon$.
\end{remark}

\begin{remark} \label{graph-mu}
When $\Sigma$ is a graph and $\Upsilon$ is a spanning forest, $\mu_\Upsilon$ is just the number of vertices of~$\Sigma$.  Then, for any edge~$\sigma$ in~$\Upsilon$, the vector $\calcut{\Upsilon}{\sigma}$ is the usual characteristic vector of the fundamental bond $\bond(\Upsilon,\sigma)$.
\end{remark}

\begin{remark} \label{dual-CMTT}
Taking $\Upsilon=\Upsilon'$ in the calculation of Proposition~\ref{calculate-L}
gives the equality
$$\sum_\Gamma\torh_{d-1}(\Sigma,\Gamma)^2
 =\frac{\torh_{d-1}(\Sigma)^2}{\torh_{d-1}(\Upsilon)^2}\det \Ldu_\Upsilon$$
which can be viewed as a dual form of Proposition~\ref{CMFT}, 
enumerating relatively acyclic $(d-1)$-subcomplexes, rather than cellular spanning forests.
\end{remark}

\section{The Flow Space}\label{flowspace-section}
In this section we describe the flow space of a cell complex.  We begin by observing that the cut and flow spaces are orthogonal to each other. 

\begin{proposition} \label{cut-flow-ortho}
The cut and flow spaces are orthogonal complements under the standard inner product on $C_d(\Sigma;\Rr)$.
\end{proposition}

\begin{proof}
First, we show that the cut and flow spaces are orthogonal.
Let $\alpha\in\Cut_d=\im\cbd_d$ and $\beta\in\Flow_d=\ker\bd_d$.  Then $\alpha = \cbd \gamma$ for some ($d-1$)-chain $\gamma$, and
$\langle \alpha,\beta\rangle = \langle \cbd\gamma,\beta\rangle = \langle \gamma,\bd\beta\rangle = 0$
by equation~\eqref{scalar-product-identity}.

It remains to show that $\Cut_d$ and $\Flow_d$ have complementary
dimensions.  Indeed, let $n=\dim C_d(\Sigma;\Rr)$; then $\dim\Flow_d =
\dim\ker\bd_d = n-\dim\im\bd_d = n-\dim\im\cbd_d = n-\dim\Cut_d$.
\end{proof}

Next we construct a basis of the flow space whose elements correspond to fundamental circuits of a given cellular spanning forest.  Although cuts and flows are in some sense dual constructions, it is easier in this case to work with kernels than images, essentially because of Proposition~\ref{TFAE-lattices}.  As a consequence, we can much more directly obtain a characteristic flow vector whose coefficients carry topological meaning.

We need one preliminary result from linear algebra.

\begin{proposition} \label{flow}
Let $N$ be an $r\x c$ integer matrix of rank $c-1$ such that
every set of $c-1$ columns is linearly independent, so that $r\geq c-1$ and $\dim\ker N=1$.
Then $\ker N$ has a spanning vector $v=(v_1,\dots,v_c)$ such that
  $$v_i=\pm|\tor(\coker N_{\bar{\imath}})|$$
where $N_{\bar{\imath}}$ denotes the submatrix of $N$ obtained by deleting
the $i^{th}$ column.  In particular, $v_i\neq 0$ for all~$i$.
\end{proposition}

\begin{proof}
Let $Q$ be an $r\x r$ matrix whose first $r-(c-1)$ rows form a $\Zz$-module basis
for $\ker(N^T)$, and whose remaining $c-1$ rows extend it to a basis
of $\Zz^r$ (see Proposition \ref{TFAE-lattices}).  Then $Q$ is invertible over $\Zz$, and the matrix
$P=QN=(N^TQ^T)^T$ has the form
  $$P = \left[\begin{array}{c}0\\\hline M\end{array}\right]$$
where $M$ is a $(c-1)\x c$ matrix whose column matroid
is the same as that of $N$.  Then $\ker N = \ker P = \ker M$.
Meanwhile, by Cramer's rule, $\ker M$ is the one-dimensional space
spanned by $v=(v_1,\dots,v_c)$, where
$v_i=(-1)^i\det M_{\bar{\imath}}=\pm|\coker M_{\bar{\imath}}|=\pm|\tor(\coker P_{\bar{\imath}})|$.
Since $Q$ is invertible, it induces isomorphisms $\coker N_{\bar{\imath}}\isom\coker QN_{\bar{\imath}}=\coker P_{\bar{\imath}}$ for all~$i$, so
$|\tor(\coker P_{\bar{\imath}})|=|\tor(\coker N_{\bar{\imath}})|$, completing the proof.
\end{proof}

Recall that a set of facets $C\subseteq \Sigma_d$ is a circuit of the cellular matroid $\MM(\Sigma)$ if and only if it corresponds to a minimal linearly dependent set of columns of~$\bd_d$.
Applying Proposition~\ref{flow} with $N=\bd_C$
(i.e., the restriction of $\bd$ to the columns indexed by~$C$), we obtain a flow vector whose support is exactly~$C$.
We call this the \emph{characteristic vector} $\uncalflow(C)$.

\begin{theorem}\label{charflow}
Let $C$ be a circuit of the cellular matroid $\MM(\Sigma)$, and let $\Delta\subseteq\Sigma$ be the subcomplex $\Sigma_{(d-1)}\cup C$.  Then
$$\uncalflow(C)=\sum_{\sigma \in C} \pm \torh_{d-1}(\Delta\sm\sigma) \, \sigma.$$
\end{theorem}

\begin{proof}
Let $N=\bd_C$, and for $\sigma\in C$, let $N_{\bar\sigma}$ denote $N$ with the column~$\sigma$ removed.
By Proposition~\ref{flow}, it suffices to show that the two groups
$$\HH_{d-1}(\Delta\sm\sigma;\Zz)=\frac{\ker \bd_{d-1}}{\im N_{\bar{\sigma}}},\qquad
\coker N_{\bar{\sigma}} = \frac{C_{d-1}(\Sigma;\Zz)}{\im N_{\bar{\sigma}}}$$
have the same torsion summands.  But this is immediate because $\ker\bd_{d-1}$ is a summand of $C_{d-1}(\Sigma;\Zz)$
as a free $\Zz$-module.
\end{proof}

\begin{example}
Consider the cell complex $\Sigma$ with two vertices $v_1$ and $v_2$,
three one-cells $e_1$, $e_2$, and $e_3$, each one with endpoints $v_1$ and $v_2$, 
and three two-cells 
$\sigma_1$, $\sigma_2$, and $\sigma_3$ attached to the one-cells so that 
the 2-dimensional boundary matrix is
$$
\bd = \bordermatrix{&\sigma_1 & \sigma_2 & \sigma_3 \cr
e_1 & 2 &  2 &  0 \cr
e_2 & -2 &  0 &  1 \cr
e_3 & 0 & -2 & -1
}.
$$
The only circuit in $\Sigma$ is the set $C$ of all three two-cells.  Thus
$\uncalflow(C) = 2\sigma_1-2\sigma_2+4\sigma_3$, because the relevant integer homology groups are
$\HH_1(\Delta\sm \sigma_1) \isom \HH_1(\Delta\sm\sigma_2) \isom \Zz_2$, but
$\HH_1(\Delta\sm \sigma_3) \isom \Zz_2 \oplus \Zz_2$.
\end{example}

For a cellular spanning forest $\Upsilon$ and facet $\sigma\not\in\Upsilon$, 
let $\circuit(\Upsilon,\sigma)$ denote 
the \emph{fundamental circuit of $\sigma$ with respect to $\Upsilon$}, 
that is, the unique circuit in $\Upsilon\cup\sigma$.

\begin{theorem} \label{flow-basis}
Let $\Sigma$ be a cell complex and $\Upsilon\subseteq\Sigma$ a cellular spanning forest.  Then the set
$$ \{\uncalflow(\circuit(\Upsilon,\sigma)) \st \sigma \not\in \Upsilon \}$$
forms an $\mathbb{R}$-vector space basis for the flow space of $\Sigma$.
\end{theorem}

\begin{proof}
The flow space is the kernel
of a matrix with $|\Sigma_d|$ columns and rank $|\Upsilon_d|$,
so its dimension is $|\Sigma_d|-|\Upsilon_d|$.  Therefore, it is enough to show that
the $\uncalflow(\circuit(\Upsilon,\sigma))$ are linearly independent.  Indeed, consider the matrix $W$ whose
rows are the vectors $\uncalflow(\circuit(\Upsilon,\sigma))$; its maximal square submatrix $W'$ whose columns correspond to $\Sigma\sm\Upsilon$
has nonzero entries on the diagonal but zeroes elsewhere.
\end{proof}

\begin{example}
Recall the bipyramid of Example~\ref{bipyramid-cut}, and its spanning tree $\Upsilon$.
Then $\circuit(\Upsilon, 125) = \{125, 123, 135, 235\}$, and 
$\circuit(\Upsilon, 134) = \{134, 123, 124, 234\}$.
If we instead consider the spanning tree $\Upsilon' = \{124, 125, 134, 135, 235\}$, 
then $\circuit(\Upsilon', 123) = \{123, 125, 135, 235\}$, and 
$\circuit(\Upsilon', 234) = \{234, 124, 125, 134, 135, 235\}$.
Each of these circuits is homeomorphic to a 2-sphere,
and the corresponding flow vectors are the homology classes they
determine.  
Furthermore, each of $\{\uncalflow(\circuit(\Upsilon, 125)), \uncalflow(\circuit(\Upsilon, 134))\}$ and 
$\{\uncalflow(\circuit(\Upsilon', 123)), \uncalflow(\circuit(\Upsilon', 234))\}$ is a basis of the flow space.
\end{example}

\section{Integral Bases for the Cut and Flow Lattices} \label{integral-bases-section}

Recall that the \emph{cut lattice} and \emph{flow lattice} of~$\Sigma$ are defined as
$$
\CC = \CC(\Sigma) = \im_\Zz\cbd_d \subseteq \Zz^n, \qquad \FF = \FF(\Sigma) = \ker_\Zz\bd_d \subseteq \Zz^n.
$$
In this section, we study the conditions under which the vector space bases of Theorems~\ref{calibration}
and~\ref{flow-basis} are integral bases for the cut and flow lattices respectively.

\begin{theorem} \label{integral-basis-cut}
Suppose that $\Sigma$ has a cellular spanning forest~$\Upsilon$
such that $\HH_{d-1}(\Upsilon;\Zz)$ is torsion-free.  Then
$$\{\calcutta{\Upsilon}{\sigma}\st \sigma\in\Upsilon\}$$
is an integral basis for the cut lattice~$\CC(\Sigma)$,
where $\calcutta{\Upsilon}{\sigma}$ is defined as in equation~\eqref{calcut-upsilon-sigma}.
\end{theorem}

\begin{proof}
Consider the $n\x r$ matrix with columns $\calcutta{\Upsilon}{\sigma}$
for $\sigma\in\Upsilon_d$.
Its restriction to the rows $\Upsilon_d$ is diagonal, and by Theorem~\ref{cut-basis}
and the hypothesis on $\HH_{d-1}(\Upsilon;\Zz)$, its entries are all $\pm1$.  Therefore, the $\calcutta{\Upsilon}{\sigma}$
form an integral basis for the lattice $\Cut_d(\Sigma)\cap\Zz^n$.
Meanwhile,
$$(\Cut_d(\Sigma)\cap\Zz^n)/\CC_d(\Sigma)=\tor(\HH^d(\Sigma;\Zz))\isom\tor(\HH_{d-1}(\Sigma;\Zz))$$
where the first equality is because $\Cut_d(\Sigma)\cap\Zz^n$ is a summand of $\Zz^n$, 
and the second one is equation~\eqref{UCTC}.
On the other hand, $\HH_{d-1}(\Sigma;\Zz)$ is a quotient of $\HH_{d-1}(\Upsilon;\Zz)$ of equal rank; 
in particular, $\tor(\HH_{d-1}(\Sigma;\Zz))=0$ and
in fact $\Cut_d(\Sigma)\cap\Zz^n=\CC_d(\Sigma)$.
\end{proof}

Next we consider integral bases of the flow lattice.
For a circuit~$C$, define
$$\calflow(C)=\frac1g\uncalflow(C)$$
where $\uncalflow(C)$ is the characteristic vector
defined in Section~\ref{flowspace-section}
and $g$ is the gcd of its coefficients.  Thus $\calflow(C)$
generates the rank-1 free $\Zz$-module of flow vectors supported
on $C$.

\begin{theorem}\label{integral-basis-flow}
Suppose that $\Sigma$ has a cellular spanning forest $\Upsilon$ such that $\HH_{d-1}(\Upsilon;\Zz)=\HH_{d-1}(\Sigma;\Zz)$.
Then $\{\calflow(\circuit(\Upsilon,\sigma))\st \sigma\not\in\Upsilon\}$ is an integral basis for the
flow lattice~$\FF(\Sigma)$.
\end{theorem}

\begin{proof}
By the hypothesis on~$\Upsilon$, the columns of~$\bd$ indexed by the facets in~$\Upsilon$ form a $\Zz$-basis for the column space.
That is, for every $\sigma\not\in\Upsilon$, the column~$\bd_\sigma$ is
a $\Zz$-linear combination of the columns of~$\Upsilon$; equivalently, there is an
element $w_\sigma$ of the flow lattice, with support
$\circuit(\Upsilon,\sigma)$, whose coefficient in the~$\sigma$ position is $\pm 1$.  But then~$w_\sigma$ and
$\calflow(\circuit(\Upsilon,\sigma))$ are integer vectors with the same linear span, both of which have the gcd of their entries equal to 1;
therefore, they must be equal up to sign.  Therefore, retaining the notation of Theorem~\ref{flow-basis}, the matrix~$W'$
is in fact the identity matrix, and it follows that the lattice spanned by the $\calflow(\circuit(\Upsilon,\sigma))$ is saturated, so it must equal
the flow lattice of~$\Sigma$.
\end{proof}

If $\Sigma$ is a graph, then all its subcomplexes and relative complexes are torsion-free
(equivalently, its incidence matrix is totally unimodular).  Therefore, Theorems~\ref{integral-basis-cut} and~\ref{integral-basis-flow}
give integral bases for the cut and flow lattices respectively.  These are, up to sign, the integral bases
constructed combinatorially in, e.g., \cite[Chapter~14]{GodRoy}.

\section{Groups and Lattices} \label{groups-lattices}

In this section, we define the critical, cocritical, and cutflow groups of a cell complex.  We identify
the relationships between these groups and to the discriminant groups of the cut and flow lattices.  The case of a graph was studied in detail by
Bacher, de~la~Harpe and Nagnibeda~\cite{BHN} and Biggs~\cite{Biggs-Chip},
and is presented concisely in~\cite[Chapter~14]{GodRoy}.

Throughout this section, let~$\Sigma$ be a cell complex of dimension~$d$ with $n$ facets,
and identify both $C_d(\Sigma;\Zz)$ and $C^d(\Sigma;\Zz)$ with $\Zz^n$.

\begin{definition}
The \emph{critical group} of~$\Sigma$ is
\begin{equation*} 
K(\Sigma):=\tor(\ker\bd_{d-1}/\im\bd_{d}\cbd_{d})=\tor(\coker(\bd_{d}\cbd_{d})).
\end{equation*}
Here and henceforth, all kernels and images are taken over~$\Zz$.
\end{definition}
Note that the second and third terms in the definition are equivalent because
$\ker\bd_{d-1}$ is a summand of $C_{d-1}(\Sigma;\Zz)$ as a free $\Zz$-module.
This definition coincides with the usual definition of the critical group of a graph in the case $d=1$, and with
the authors' previous definition in \cite{Critical} in the case that $\Sigma$ is $\Qq$-acyclic in codimension~one
(when $\ker\bd_{d-1}/\im\bd_{d}\cbd_{d}$ is its own torsion summand).

\begin{definition}
The \emph{cutflow group} of $\Sigma$ is $\Zz^n/(\CC(\Sigma)\oplus\FF(\Sigma))$.
\end{definition}

Note that the cutflow group is finite because the cut and flow spaces are orthogonal
complements in~$\Rr^n$ (Proposition~\ref{cut-flow-ortho}), so in particular $\CC\oplus\FF$ spans~$\Rr^n$ as a vector space.
Observe also that the cutflow group does not decompose into separate cut and flow pieces; that is, it is not isomorphic to
the group $G=((\Flow_d\cap \Zz^n)/\ker\bd_d)\oplus((\Cut_d\cap \Zz^n)/\im\cbd_d),$
even when $\Sigma$ is a graph.  For example, if $\Sigma$ is the complete graph on three vertices, whose boundary map can be written as
$$\bd=\begin{pmatrix} 1&-1&0\\-1&0&1\\0&1&-1\end{pmatrix},$$
then $\ker\bd=\Span\{(1,1,1)^T\}$ and $\im\cbd=\Span\{(1,-1,0)^T,(1,0,-1)^T\}$.
So $G$ is the trivial group, while $\Zz^n / ( \ker\bd_d\oplus\im\cbd_d )=K_1(\Sigma)\isom\Zz_3$.

In order to define the cocritical group of a cell complex, we first need to introduce the notion of acyclization.

\begin{definition} \label{acyclization}
An \emph{acyclization} of~$\Sigma$ is a ($d+1)$-dimensional
complex $\Omega$ such that $\Omega_{(d)}=\Sigma$ and
$\HH_{d+1}(\Omega;\Zz)=\HH_d(\Omega;\Zz)=0$.
\end{definition}

Algebraically, this construction corresponds to
finding an integral basis for $\ker\bd_d(\Sigma)$ and declaring its elements to be the columns
of $\bd_{d+1}(\Omega)$ (so in particular $|\Omega_{(d+1)}|=\tilde\beta_d(\Sigma)$).  
Topologically, it corresponds to filling in just enough $d$-dimensional cycles with $(d+1)$-dimensional faces to remove all $d$-dimensional homology.
The definition of
acyclization and equation~\eqref{UCTC} together imply
that $\HH^{d+1}(\Omega;\Zz) = 0$; that is, $\cbd_{d+1}(\Omega)$ is surjective.

\begin{definition}
The \emph{cocritical group} $K^*(\Sigma)$ is
\begin{equation*}
K^*(\Sigma) := C_{d+1}(\Omega;\Zz)/\im\cbd_{d+1}\bd_{d+1}=\coker\Ldu_{d+1}.
\end{equation*}
\end{definition}

It is not immediate that the group $K^*(\Sigma)$ is independent of the choice of~$\Omega$;
we will prove this independence as part of Theorem~\ref{new-flow-theorem}.
For the moment, it is at least clear that $K^*(\Sigma)$ is finite, since $\rank\cbd_{d+1}=\rank\Ldu_{d+1}=\rank C_{d+1}(\Omega;\Zz)$.  
%
In the special case of a graph,  $\Ldu_{d+1}$ is the ``intersection matrix'' defined by Kotani and Sunada~\cite{KS}.
(See also~\cite[Sections 2, 3]{Biggs-Crypto}.)  

\begin{remark} \label{poincare}
As in \cite{Critical}, one can define critical and cocritical groups in every dimension by
$K_i(\Sigma)=\tor(C_i(\Sigma;\Zz)/\im\bd_{i+1}\cbd_{i+1})$ 
and
$K^*_i(\Sigma)= \tor(C_i(\Sigma;\Zz)/\im\cbd_i\bd_i)$.
If the cellular chain complexes of $\Sigma$ and $\Psi$ are algebraically dual (for example, if $\Sigma$ and $\Psi$  are Poincar\'e dual cell structures on a compact orientable $d$-manifold),
then $K_i(\Psi)=K_{d-i}^*(\Sigma)$ for all $i$.
\end{remark}

We now come to the main results of the second half of the paper:
the critical and cocritical groups are isomorphic to the discriminant groups
of the cut and flow lattices respectively, and the cutflow group mediates between the
critical and cocritical groups, with an ``error term'' given by homology.

\begin{theorem} \label{new-cut-theorem}
Let~$\Sigma$ be a cell complex of dimension~$d$ with $n$ facets.  Then there is a commutative diagram
\begin{equation} \label{cut-CD}
\xymatrix{
0 \rto & \Zz^n/(\CC\oplus\FF) \dto^{\alpha}\rto^{\psi} & \CC^\sharp/\CC\dto^{\beta}\rto & \tor(\HH^d(\Sigma;\Zz))\dto^{\gamma}\rto & 0 \\
0 \rto & \im \bd_d /\im \bd_d\cbd_d  \rto & K(\Sigma)\rto & \tor(\HH_{d-1}(\Sigma;\Zz))\rto & 0 }
\end{equation}
in which all vertical maps are isomorphisms. In particular, $K(\Sigma) \cong \CC^\sharp/\CC$.
\end{theorem}

\begin{proof}
\emph{Step 1: Construct the bottom row of~\eqref{cut-CD}.}
The inclusions $\im\bd_d\cbd_d\subseteq\im\bd_d\subseteq\ker\bd_{d-1}$
give rise to the short exact sequence
\begin{equation*} 
0 \to \im \bd_d/\im \bd_d\cbd_d \to \ker \bd_{d-1}/\im \bd_d\cbd_d \to \ker\bd_{d-1}/\im\bd_d\to 0.
\end{equation*}
The first term is finite (because $\rank\bd_d=\rank\bd_d\cbd_d$), so taking torsion summands yields the desired short exact sequence.

\emph{Step 2: Construct the top row of~\eqref{cut-CD}.}
Let $r=\rank\bd_d$, let $\{v_1,\dots,v_r\}$ be an integral
basis of~$\CC$, and let $V$ be the matrix with columns $v_1,\dots,v_r$.
By Proposition~\ref{lattice-projection}, 
the dual basis $\{v_1^*,\dots,v_r^*\}$ for~$\CC^\sharp$ consists of the columns of
the matrix $W=V(V^TV)^{-1}$.  Let $\psi$ be the orthogonal projection $\Rr^n\to\Cut(\Sigma)$,
which is given by the matrix $P=WV^T=V(V^TV)^{-1}V^T$ (see Proposition~\ref{lattice-projection}).  Then
  $$\im\psi=\colspace\left(\underbrace{\big[v_1^*\cdots v_r^*\big]}_{W}\ \underbrace{\big[v_1\cdots v_r\big]^T}_{V^T}\right).$$
The $i^{th}$ column of $P$ equals $W$ times the $i^{th}$ column of $V^T$.
If we identify $\CC^\sharp$ with $\Zz^r$ via the basis
$\{v_1^*,\dots,v_r^*\}$, then $\im(\psi)$ is just the column space of $V^T$.
So $\CC^\sharp/\im\psi\isom\Zz^r/\colspace(V^T)$,
which is a finite group because $\rank V=r$.  Since the matrices $V$ and $V^T$ have the same invariant factors,
we have
  $$\Zz^r/\colspace(V^T)\isom\tor(\Zz^n/\colspace(V))=\tor(C^d(\Sigma;\Zz)/\im\cbd_d)=\tor(\HH^d(\Sigma;\Zz)).$$
Meanwhile, $\im \psi \supseteq\CC$ because $PV=V$.
Since $\ker \psi =\FF$, we have
$(\im \psi)/\CC=(\Zz^n/\FF)/\CC=\Zz^n/(\CC\oplus\FF)$.
Therefore, the inclusions
$\CC\subseteq\im \psi \subseteq\CC^\sharp$ give rise to the short exact sequence
in the top row of~\eqref{cut-CD}.

\emph{Step 3: Describe the vertical maps in~\eqref{cut-CD}.}
The maps $\alpha$ and $\beta$ are each induced by $\bd_d$ in the following ways.
First, the image of the cutflow group under~$\bd_d$ is
$$\bd_d\left(\Zz^n/(\FF\oplus\CC)\right)
=\bd_d\left(\Zz^n/(\ker\bd_d\oplus\im\cbd_d)\right)
= \im\bd_d/\im\bd_d\cbd_d.$$
On the other hand, $\bd_d$ acts injectively on the cutflow group (since the latter is a subquotient of $\Zz^n/\ker\bd_d$).
So the map labeled $\alpha$ is an isomorphism.

The cellular boundary map $\bd_d$ also gives rise to the map $\beta\colon \CC^\sharp/\CC\to K(\Sigma)$, as we now explain.
First, note that $\bd_d\,\CC^\sharp\subseteq\im_\Rr\bd_d\subseteq\ker_\Rr\bd_{d-1}$.
Second, observe that for every $w\in\CC^\sharp$ and $\rho\in C_{d-1}(\Sigma;\Zz)$, we have $\scp{\bd w,\rho}=\scp{w,\cbd\rho}\in\Zz$,
by equation~\eqref{scalar-product-identity} and the definition of dual lattice.
Therefore, $\bd_d\,\CC^\sharp\subseteq C_{d-1}(\Sigma;\Zz)$.  It follows that
$\bd_d$ maps $\CC^\sharp$ to $(\ker_\Rr\bd_{d-1})\cap C_{d-1}(\Sigma;\Zz)=\ker_\Zz\bd_{d-1}$,
hence defines a map $\beta\colon \CC^\sharp/\CC\to\ker_\Zz\bd_{d-1}/\im_\Zz \bd_d\cbd_d$.
Since $\CC^\sharp/\CC$ is finite, the image of $\beta$ is purely torsion, hence contained in $K(\Sigma)$.
Moreover, $\beta$ is injective because $(\ker\bd_d)\cap\CC^\sharp=\FF\cap\CC^\sharp=0$ by Proposition~\ref{cut-flow-ortho}.

Every element of $\Rr^n$ can be written uniquely as $c+f$ with
$c\in\Cut(\Sigma)$ and $f\in\Flow(\Sigma)$.  The map $\psi$ is
orthogonal projection onto $\Cut(\Sigma)$, so $\bd_d(c+f)=\bd_d
c=\bd_d(\psi(c+f))$.  Hence the left-hand square commutes.  The map
$\gamma$ is then uniquely defined by diagram-chasing.

The snake lemma now implies that $\ker\gamma=0$.  Since the groups
$\tor(\HH_{d-1}(\Sigma;\Zz))$ and $\tor(\HH^d(\Sigma;\Zz))$ are
abstractly isomorphic by equation~\eqref{UCTC}, in fact $\gamma$ must be an isomorphism and
$\coker\gamma=0$ as well.  Applying the snake lemma again, we see
that all the vertical maps in~\eqref{cut-CD} are isomorphisms.
\end{proof}

\begin{theorem} \label{new-flow-theorem}
Let~$\Sigma$ be a cell complex of dimension~$d$ with $n$ facets.  Then there is a short exact sequence
\begin{equation} \label{flow-SES}
0 \to \tor(\HH_{d-1}(\Sigma;\Zz)) \to \Zz^n/(\CC\oplus\FF)\to \FF^\sharp/\FF\to 0.
\end{equation}
Moreover, $K^*(\Sigma)\isom\FF^\sharp/\FF$.
\end{theorem}
\begin{proof}
Let $\Omega$ be an acyclization of $\Sigma$.  By construction, the columns of the matrix $A$ representing $\bd_{d+1}(\Omega)$
form an integral basis for $\FF=\ker\bd_d$.  Again,
the matrix $Q=A(A^TA)^{-1}A^T$ represents orthogonal projection
$\Rr^n\to\Flow(\Sigma)$.  The maximal minors of $A$ have gcd 1 (because $\FF$ is a summand of $\Zz^n$, 
so the columns of $A$ are part of an integral basis), so by 
Proposition~\ref{lattice-projection}, 
the columns of $Q$ generate the lattice $\FF^\sharp$.  Therefore,
if we regard $Q$ as a map of $\Zz$-modules, it defines a surjective homomorphism $\Zz^n\to\FF^\sharp$.
This map fixes~$\FF$ pointwise and its kernel is the saturation~$\hat\CC:=(\CC\otimes\Rr)\cap\Zz^n$.
So we have short exact sequences
$0 \to \hat\CC \to \Zz^n/\FF\to\FF^\sharp/\FF \to 0$ and
$$
0 \to \hat\CC/\CC \to \Zz^n/(\CC\oplus\FF)\to\FF^\sharp/\FF \to 0.
$$
Since $\hat\CC$ is a summand of $\Zz^n$ by Proposition~\ref{TFAE-lattices}, we can identify $\hat\CC/\CC$ with $\tor(\HH^d(\Sigma;\Zz))
\isom\tor(\HH_{d-1}(\Sigma;\Zz))$,
 which gives the short exact sequence~\eqref{flow-SES}.

We will now show that $\FF^\sharp/\FF\isom K^*(\Sigma)$.  To see this, observe that
$\cbd_{d+1}(\FF^\sharp)=\cbd_{d+1}(\colspace(Q))=\colspace(A^TQ)=\colspace(A^T)=\im\cbd_{d+1}=C_{d+1}(\Omega)$
(by the construction of an acyclization).
In addition,
$\ker\cbd_{d+1}$ is orthogonal to $\FF^\sharp$, hence their intersection is zero.
Therefore, $\cbd_{d+1}$ defines an isomorphism $\FF^\sharp\to C_{d+1}(\Omega)$.
Moreover, the same map $\cbd_{d+1}$ maps $\FF=\ker\bd_d=\im\bd_{d+1}$ surjectively onto 
$\im\cbd_{d+1}\bd_{d+1}$.
\end{proof}

\begin{corollary} \label{if-no-torsion}
If $\HH_{d-1}(\Sigma;\Zz)$ is torsion-free, then the groups
$K(\Sigma)$, $K^*(\Sigma)$, $\CC^\sharp/\CC$, $\FF^\sharp/\FF$,
and $\Zz^n/(\CC\oplus\FF)$ are all isomorphic to each other.
\end{corollary}

Corollary~\ref{if-no-torsion} includes the case that $\Sigma$ is a
graph, as studied by Bacher, de~la~Harpe and Nagnibeda~\cite{BHN} and
Biggs~\cite{Biggs-Chip}.  It also includes the combinatorially important
family of Cohen-Macaulay (over $\Zz$) simplicial complexes, as well as
cellulations of compact orientable manifolds.

\begin{example}
Suppose that $\HH_d(\Sigma;\Zz)=\Zz$ and that $\HH_{d-1}(\Sigma;\Zz)$
is torsion-free.  Then the flow lattice is generated by a single
element, and it follows from Corollary~\ref{if-no-torsion} that
$K(\Sigma)\isom K^*(\Sigma)\isom\FF^\sharp/\FF$ is a cyclic group.
For instance, if $\Sigma$ is homeomorphic to a cellular sphere or torus,
then the critical group is cyclic of order equal to the number of
facets.  (The authors had previously proved this fact for simplicial
spheres~\cite[Theorem~3.7]{Critical}, but this approach
using the cocritical group makes the statement more general
and the proof transparent.)
\end{example}

\begin{example} \label{RPtwo}
Let $\Sigma$ be the standard cellulation $e^0\cup e^1\cup e^2$
of the real projective plane,
whose cellular chain complex is
$$\Zz~\xrightarrow{\bd_2=2}~\Zz~\xrightarrow{\bd_1=0}~\Zz.$$
Then $\CC=\im\cbd_2=2\Zz$, $\CC^\sharp=\frac12\Zz$, and
$K(\Sigma)=\CC^\sharp/\CC=\Zz_4$.  Meanwhile,
$\FF=\FF^\sharp=\FF^\sharp/\FF=K^*(\Sigma)=0$.
The cutflow group is $\Zz_2$.  Note that the rows of Theorem~\ref{new-cut-theorem} are not split in this case.
\end{example}

\begin{example} \label{Vic}
Let $a,b\in\Zz\sm\{0\}$.
Let $\Sigma$ be the cell complex whose cellular chain complex is
$$\Zz^2~\xrightarrow{\bd_2=[a\ b]}~\Zz~\xrightarrow{\bd_1=0}~\Zz.$$
Topologically, $\Sigma$ consists of a vertex $e^0$, a loop $e^1$, and two facets
of dimension~2 attached along $e^1$ by maps of degrees~$a$ and $b$.
Then
$$\CC^\sharp/\CC=\Zz_\tau,\qquad \Zz^2/(\CC\oplus\FF)=\Zz_{\tau/g},\qquad
\FF^\sharp/\FF=\Zz_{\tau/g^2},$$
where $\tau=a^2+b^2$ and $g=\gcd(a,b)$. Note that $\tau=\tau_2(\Sigma)$ is the complexity of $\Sigma$ (see equation~\eqref{complexity})
and that $g=|\HH_1(\Sigma;\Zz)|$.
The short exact sequence of Theorem~\ref{new-flow-theorem}
is in general not split (for example, if $a=6$ and $b=2$).
\end{example}

\section{Enumeration}\label{sec:sizes}

For a connected graph, the cardinality of the critical group equals the number of spanning trees.
In this section, we calculate the cardinalities of the
various group invariants of~$\Sigma$.

Examples~\ref{RPtwo} and~\ref{Vic} both
indicate that $K(\Sigma)\isom\CC^\sharp/\CC$ should have cardinality
equal to the complexity $\tau(\Sigma)$.
Indeed, in Theorem~4.2 of~\cite{Critical}, the authors proved that $|K(\Sigma)|=\tau(\Sigma)$ whenever $\Sigma$ has a cellular spanning tree $\Upsilon$ such that $\HH_{d-1}(\Upsilon;\Zz)=\HH_{d-1}(\Sigma;\Zz)=0$ (in particular, $\Sigma$ must be not merely
$\Qq$-acyclic, but actually $\Zz$-acyclic, in codimension~one).
Here we prove that this condition is actually not necessary: for any cell complex,
the order of the critical group $K(\Sigma)$
equals the torsion-weighted complexity $\tau(\Sigma)$.  Our approach is to determine the size of the discriminant group
$\CC^\sharp/\CC$ directly, then use the short exact sequences of Theorems~\ref{new-cut-theorem} and~\ref{new-flow-theorem} to calculate
the sizes of the other groups.

\begin{theorem} \label{all-counts}
Let $\Sigma$ be a $d$-dimensional cell complex and 
let $\torh=\torh_{d-1}(\Sigma)=|\tor(\HH_{d-1}(\Sigma;\Zz))|$.  Then
\begin{align*}
|\CC^\sharp/\CC| = |K(\Sigma)| &= \tau_d(\Sigma),\\
|\Zz^n/(\CC\oplus\FF)| &= \tau_d(\Sigma)/\torh, \ \ \text{and}\\
|\FF^\sharp/\FF| = |K^*(\Sigma)| &= \tau_d(\Sigma)/\torh^2.
\end{align*}
\end{theorem}

\begin{proof}
By Theorems~\ref{new-cut-theorem} and~\ref{new-flow-theorem}, it is enough to prove that $|\CC^\sharp/\CC|=\tau_d(\Sigma)$.

Let $R$ be a set of $(d-1)$-cells corresponding to a row basis for $\bd$
(hence a vector space basis for $\Cut(\Sigma)$); let $\RR$ be the lattice spanned by those rows (which is a full-rank integral sublattice of $\CC$);
and let $\Gamma=(\Sigma_{d-1}\sm R)\cup\Sigma_{(d-2)}$.
The inclusions $\RR\subseteq\CC\subseteq C^d(\Sigma)$ give rise to a short exact sequence
$0\to\CC/\RR\to\HH^d(\Sigma,\Gamma;\Zz)\to\HH^d(\Sigma;\Zz)\to0$.
Since $\CC/\RR$ is finite, the torsion summands form
a short exact sequence (see Section~\ref{cellular-forest}).
Taking cardinalities and using equation~\eqref{UCTC}, we get
\begin{equation} \label{CC-LL}
|\CC/\RR|=
\frac{|\tor(\HH^d(\Sigma,\Gamma;\Zz))|}{|\tor(\HH^d(\Sigma;\Zz))|}=\frac{\torh_{d-1}(\Sigma,\Gamma)}{\torh_{d-1}(\Sigma)}.
\end{equation}
The inclusions $\RR\subseteq\CC\subseteq\CC^\sharp\subseteq\RR^\sharp$ give
$|\RR^\sharp/\RR| = |\RR^\sharp/\CC^\sharp| \cdot|\CC^\sharp/\CC| \cdot|\CC/\RR|$.
Moreover, $\RR^\sharp/\CC^\sharp\isom\CC/\RR$.  By equation~\eqref{CC-LL} and Binet-Cauchy, we have
\begin{align*}
|\CC^\sharp/\CC| &= \frac{|\RR^\sharp/\RR|}{|\RR^\sharp/\CC^\sharp| \cdot|\CC/\RR|}
= \frac{\torh_{d-1}(\Sigma)^2}{\torh_{d-1}(\Sigma,\Gamma)^2}\ |\RR^\sharp/\RR|\\
&= \frac{\torh_{d-1}(\Sigma)^2}{\torh_{d-1}(\Sigma,\Gamma)^2}\ \det(\bd_R\cbd_R)\\
&= \sum_{\Upsilon\subseteq \Sigma_d\st |\Upsilon|=r} \frac{\torh_{d-1}(\Sigma)^2}{\torh_{d-1}(\Sigma,\Gamma)^2}\ \det(\bd_{R,\Upsilon})^2.
\end{align*}
By Proposition~\ref{det-is-homology}, the summand is nonzero if and only if $\Upsilon$ is a cellular spanning forest.
In that case, the matrix $\bd_{R,\Upsilon}$ is the cellular boundary matrix of the relative complex $(\Upsilon,\Gamma)$,
and its determinant is (up to sign) $\torh_{d-1}(\Upsilon,\Gamma)$, so
by Proposition~\ref{relative-tor-formula} we have
$$
|\CC^\sharp/\CC|
~= \sum_\Upsilon \frac{\torh_{d-1}(\Sigma)^2}{\torh_{d-1}(\Sigma,\Gamma)^2}\ \torh_{d-1}(\Upsilon,\Gamma)^2 
~= \sum_\Upsilon \torh_{d-1}(\Upsilon)^2
$$
with the sums over all cellular spanning forests $\Upsilon\subseteq \Sigma$.
\end{proof}

Dually, we can interpret the cardinality of the cocritical group as
enumerating cellular spanning forests by relative torsion (co)homology,
as follows.

\begin{theorem} \label{cocritical-count}
Let $\Omega$ be an acyclization of $\Sigma$.  Then
$$|K^*(\Sigma)| = \sum_{ \Upsilon} |\HH^{d+1}(\Omega,\Upsilon;\Zz)|^2
= \sum_{ \Upsilon} |\HH_d(\Omega,\Upsilon;\Zz)|^2$$
with the sums over all cellular spanning forests $\Upsilon\subseteq \Sigma$.
\end{theorem}

Note that the groups $\HH^{d+1}(\Omega,\Upsilon;\Zz)$ and $\HH_d(\Omega,\Upsilon;\Zz)$
are all finite, by definition of acyclization.

\begin{proof}
Let $\bd_{d+1}=\bd_{d+1}(\Omega)$.
Note that $\rank\bd_{d+1}=\betti_d(\Sigma)$; abbreviate this number as $b$.
By Binet-Cauchy, we have
$$|K^*(\Sigma)| = |\det\cbd_{d+1}\bd_{d+1}|
= \sum_{B\subseteq\Omega_d\st |B|=b} (\det\cbd_B)^2$$
where $\bd_B$ denotes the submatrix of $\bd$ with rows $B$.
Letting $\Upsilon=\Sigma\sm B$, 
we can regard $\bd_B$ as the cellular boundary map
of the relative complex $(\Omega,\Upsilon)$, which consists
of~$b$ cells in each of the dimensions~$d$ and $d+1$.
By Proposition~\ref{det-is-homology}, the summand is nonzero
if and only if $\Upsilon$ is a cellular spanning forest
of $\Omega_{(d)}=\Sigma$.  (Note that the $d+1,B,\Upsilon,\Omega_{(d)}$
in the present context correspond respectively to the $d,R,\Gamma,\Upsilon$ of
Proposition~\ref{det-is-homology}.)  For these summands,
$\HH^{d+1}(\Omega,\Upsilon;\Zz)\cong\HH_d(\Omega,\Upsilon;\Zz)$
is a finite group of order $|\det\bd_B|$.
\end{proof}

\begin{remark}
Let $\tau^*(\Sigma)=\sum_\Upsilon |\HH_d(\Omega,\Upsilon;\Zz)|^2$, as in
Theorem~\ref{cocritical-count}.  Then combining Theorems~\ref{all-counts} and Theorems~\ref{cocritical-count} gives
\begin{alignat*}{2} 
|\CC^\sharp/\CC| = |K(\Sigma)| &=  \tau(\Sigma)  & &= \tau^*(\Sigma)\cdot\torh^2,\\
|\FF^\sharp/\FF| = |K^*(\Sigma)| &=  \tau^*(\Sigma) & &= \tau(\Sigma)/\torh^2,\\
|\Zz^n/(\CC\oplus\FF)| &=  \tau(\Sigma)/\torh  & &= \tau^*(\Sigma)\cdot\torh,
\end{alignat*}
highlighting the duality between the cut and flow lattices.
\end{remark}

\section{Bounds on combinatorial invariants from lattice geometry} \label{Hermite}

Let $n\geq 1$ be an integer.  The \emph{Hermite constant} $\gamma_n$
is defined as the maximum value of
\begin{equation} \label{herm}
\left(\min_{x\in\LL\sm\{0\}}\langle x,x\rangle\right) (|\LL^\sharp/\LL|)^{-1/n}
\end{equation}
over all lattices $\LL\subseteq\Rr^n$.  The Hermite constant arises both
in the study of quadratic forms and in sphere packing;
see~\cite[Section~4]{Lag}.  It is known that $\gamma_n$ is finite for every
$n$, although the precise values are known only for $1\leq n\leq 8$
and $n=24$~\cite{CK}.

As observed by Kotani and Sunada \cite{KS}, if $\LL=\FF$ is the flow
lattice of a connected graph, then the shortest vector in $\FF$ is the
characteristic vector of a cycle of minimum length; therefore, the
numerator in equation~\eqref{herm} is the girth of $G$.  Meanwhile,
$|\FF^\sharp/\FF|$ is the number of spanning trees.  We now generalize
this theorem to cell complexes.

\begin{definition}
Let $\Sigma$ be a cell complex.  The \emph{girth} and the
\emph{connectivity} are defined as the cardinalities of, respectively
the smallest circuit and the smallest cocircuit of the cellular
matroid of $\Sigma$.
\end{definition}

\begin{theorem}\label{Hermite-relations}
Let $\Sigma$ be a cell complex of dimension~$d$ with girth $g$
and connectivity $k$, and top boundary map of rank~$r$.  Let $b=\rank\FF(\Sigma)=\rank\HH_d(\Sigma;\Zz)$.  Then
$$k \tau(\Sigma)^{-1/r} \leq \gamma_r\qquad\text{and}\qquad g \tau^*(\Sigma)^{-1/b} \leq \gamma_b.$$
\end{theorem}

\begin{proof}
Every nonzero vector of the cut lattice $\CC$ contains a cocircuit in
its support, so $\min_{x\in\CC\sm\{0\}}\langle x,x\rangle\geq k$.  Likewise,
every nonzero vector of the flow lattice $\FF$ of $\Sigma$ contains a
circuit in its support, so $\min_{x\in\FF\sm\{0\}}\langle x,x\rangle\geq g$.
Meanwhile, $|\CC^\sharp/\CC|=\tau$ and $|\FF^\sharp/\FF|=\tau^*$
by Theorem~\ref{all-counts}.
The desired inequalities now follow from applying the definition of
Hermite's constant to the cut and flow lattices respectively.
\end{proof}

\bibliographystyle{alpha}

\begin{thebibliography}{99}


\bibitem{Artin}
Michael Artin,
\textit{Algebra},
Prentice-Hall, 1991.

\bibitem{BHN}
Roland Bacher, Pierre de~la~Harpe, and Tatiana Nagnibeda,
The lattice of integral flows and the lattice of integral cuts on a finite graph,
\emph{Bull.\ Soc.\ Math.\ France} {\bf 125} (1997), no.~2, 167--198.

\bibitem{BN}
Matthew Baker and Serguei Norine,
Riemann-Roch and Abel-Jacobi theory on a finite graph,
\emph{Adv.\ Math.}\ {\bf 215}, no.~2 (2007), 766--788.

\bibitem{Biggs-Chip}
N.L.\ Biggs,
Chip-firing and the critical group of a graph.
\emph{J. Algebraic Combin.}\ {\bf 9} (1999), no.~1, 25--45. 

\bibitem{Biggs-Crypto}
Norman Biggs,
The critical group from a cryptographic perspective,
\textit{Bull.\ London Math.\ Soc.}\ {\bf 39} (2007), no.~5, 829--836.

\bibitem{RedBook}
Anders Bj\"orner, Michel Las~Vergnas, Bernd Sturmfels, Neil White, and G\"unter Ziegler,
\emph{Oriented Matroids},
Second edition. Encyclopedia of Mathematics and its Applications, 46. Cambridge University Press, Cambridge, 1999.

\bibitem{BL}
Ben Bond and Lionel Levine,
Abelian networks: Foundations and examples, 
preprint, arXiv:1309.3445v1 [cs.FL] (2013).

\bibitem{Catanzaro}
Michael J. Catanzaro, Vladimir Y. Chernyak, and John R. Klein,
Kirchhoff's theorems in higher dimensions and Reidemeister torsion,
\emph{Homology Homotopy Appl.}, to appear;
preprint, arXiv:1206.6783v2 [math.AT] (2012).

\bibitem{CK}
Henry Cohn and Abhinav Kumar,
Optimality and uniqueness of the Leech lattice among lattices,
\textit{Ann.\ of Math.\ (2)} {\bf 170} (2009), no.~3, 1003--1050. 

\bibitem{CorLin}
Raul Cordovil and Bernt Lindstr\"om,
Simplicial matroids, in \textit{Combinatorial geometries}, 98--113, 
Encyclopedia Math.\ Appl., 29, Cambridge Univ. Press, Cambridge, 1987.

\bibitem{Moci1}
Michele D'Adderio and Luca Moci,
Arithmetic matroids, the Tutte polynomial and toric arrangements,
\emph{Adv.\ Math.} {\bf 232} (2013), 335--367.

\bibitem{Dhar}
Deepak Dhar,
Self-organized critical state of sandpile automaton models,
\emph{Phys.\ Rev.\ Lett.}\ {\bf 64} (1990), no.~14 1613--1616.

\bibitem{Denham}
Graham Denham,
The combinatorial Laplacian of the Tutte complex,
\emph{J. Algebra} {\bf 242} (2001), 160--175.

\bibitem{Dodziuk}
Jozef Dodziuk and Vijay Kumar Patodi,
Riemannian structures and triangulations of manifolds,
\emph{J.\ Indian Math.\ Soc.\ (N.S.)} {\bf 40} (1976), no.~1--4, 1--52 (1977).

\bibitem{Simplicial}
Art M.\ Duval, Caroline J.\ Klivans, and Jeremy L. Martin,
Simplicial matrix-tree theorems,
\textit{Trans.\ Amer.\ Math.\ Soc.}\ {\bf 361} (2009), 6073--6114.

\bibitem{Cellular}
Art M.\ Duval, Caroline J.\ Klivans, and Jeremy L. Martin,
Cellular spanning trees and Laplacians of cubical complexes,
\textit{Adv.\ Appl.\ Math.}\ {\bf 46} (2011), 247--274.

\bibitem{Critical}
Art M.\ Duval, Caroline J.\ Klivans, and Jeremy L.\ Martin,
Critical groups of simplicial complexes,
\emph{Ann.\ Comb.}\ {\bf 17} (2013), 53--70.

\bibitem{Eckmann}
Beno Eckmann,
Harmonische Funktionen und Randwertaufgaben in einem Komplex,
\emph{Comment.\ Math.\ Helv.}\ {\bf 17} (1945), 240--255.

\bibitem{Moci2}
Alex Fink and Luca Moci,
Matroids over a ring,
\emph{J. Eur.\ Math. Soc.}, to appear;
preprint, arXiv:1209.6571v2 [math.CO] (2012).

\bibitem{Friedman}
Joel Friedman,
Computing Betti Numbers via Combinatorial Laplacians,
\emph{Algorithmica} {\bf 21} (1998), no.~4, 331--346. 

\bibitem{GodRoy}
Chris Godsil and Gordon Royle,
\emph{Algebraic Graph Theory},
Graduate Texts in Mathematics 207, Springer-Verlag, New York, 2001.

\bibitem{HMY}
Christian Haase, Gregg Musiker, and Josephine Yu,
Linear systems on tropical curves,
\emph{Math.\ Z.} {\bf 270} (2012), nos.~3--4, 1111--1140.

\bibitem{Hatcher}
Allen Hatcher, \emph{Algebraic Topology},
Cambridge University Press, 2001.

\bibitem{Hungerford}
Thomas W.\ Hungerford,
\emph{Algebra},
Springer, Graduate Texts in Mathematics 73, New York, 1974.

\bibitem{Kalai}
Gil Kalai,
Enumeration of ${\bf Q}$-acyclic simplicial complexes,
\emph{Israel J.\ Math.}\ {\bf 45} (1983), no.~4, 337--351.

\bibitem{KS}
Motoko Kotani and Toshikazu Sunada,
Jacobian tori associated with a finite graph and its abelian covering graphs,
\emph{Adv.\ Appl.\ Math.}\ {\bf 24} (2000), no.~2, 89--110.

\bibitem{Lag}
Jeffrey Lagarias,
Point lattices. \textit{Handbook of combinatorics}, Vol.~1, 919--966, Elsevier, Amsterdam, 1995.

\bibitem{Lor}
Dino J.\ Lorenzini,
A finite group attached to the Laplacian of a graph,
\emph{Discrete Math.}\ {\bf 91} (1991), no.~3, 277--282.

\bibitem{Lyons}
Russell Lyons,
Random complexes and $\ell^2$-Betti numbers,
\emph{J.\ Topol.\ Anal.}\ {\bf 1} (2009), no.~2, 153--175.

\bibitem{Merris}
Russell Merris,
Laplacian matrices of graphs: a survey,
\emph{Linear Algebra Appl.}\ {\bf 197-198} (1994), 143--176.

\bibitem{Oxley}
James Oxley,
\emph{Matroid Theory},
Oxford University Press, New York, 1992.

\bibitem{SW}
Yi Su and David G.\ Wagner,
The lattice of integer flows of a regular matroid,
\emph{J.\ Combin.\ Theory Ser.\ B} {\bf 100} (2010), no.~6, 691--703.

\bibitem{Tutte}
W.T.\ Tutte,
Lectures on matroids,
\textit{J.\ Res.\ Nat.\ Bur.\ Standards Sect.\ B} {\bf 69B} (1965), 1--47.

\end{thebibliography}

\end{document}